%% file: manuscript-numapde-preprint.tex
\documentclass{numapde-preprint}

\usepackage{numapde-semantic}
\usepackage{numapde-style}
\usepackage{numapde-local}

\hypersetup{
	pdftitle={Fenchel Duality Theory and a Primal-Dual Algorithm on Riemannian Manifolds},
	pdfauthor={Ronny Bergmann, Roland Herzog, Maurício Silva Louzeiro, Daniel Tenbrinck, José Vidal-Núñez},
	pdfkeywords={convex analysis, Fenchel conjugate function, Riemannian manifold, Hadamard manifold, primal-dual algorithm, Chambolle--Pock algorithm, total variation}
}

\title{Fenchel Duality Theory and a Primal-Dual Algorithm on Riemannian Manifolds}
\subtitle{}
\shorttitle{Fenchel Duality on Manifolds}

\author{Ronny Bergmann\thanks{Technische Universität Chemnitz, Faculty of Mathematics, 09107 Chemnitz, Germany (\email{ronny.bergmann@math.tu-chemnitz.de}, \url{https://www.tu-chemnitz.de/mathematik/part_dgl/people/bergmann}, \orcid{0000-0001-8342-7218}, \email{roland.herzog@math.tu-chemnitz.de}, \url{https://www.tu-chemnitz.de/mathematik/part_dgl/people/herzog}, \orcid{0000-0003-2164-6575}, \email{mauricio.silva-louzeiro@math.tu-chemnitz.de}, \url{https://www.tu-chemnitz.de/mathematik/part_dgl/people/louzeiro}, \orcid{0000-0002-4755-3505}).}
\and
Roland Herzog\footnotemark[1]
\and
Maurício Silva Louzeiro\footnotemark[1]
\and
Daniel Tenbrinck\thanks{Friedrich-Alexander Universität Erlangen-Nürnberg, Department of Mathematics, Chair for Applied Mathematics (Modeling and Numerics), 91058 Erlangen, Germany (\email{daniel.tenbrinck@fau.de}, \url{https://en.www.math.fau.de/angewandte-mathematik-1/mitarbeiter/dr-daniel-tenbrinck/}).}
\and
José Vidal-Núñez\thanks{University of Alcalá, Department of Physics and Mathematics, 28801 Alcalá de Henares, Spain (\email{j.vidal@uah.es}, \url{https://www.uah.es/es/estudios/profesor/Jose-Vidal-Nunez/}, \orcid{0000-0002-1190-6700}).}}
\shortauthor{R. Bergmann, R. Herzog, M. Silva Louzeiro, D. Tenbrinck and J. Vidal-Núñez}

\dedication{}

\IfFileExists{numapde-uncolorizeHyperrefIfChanges.sty}{\RequirePackage{numapde-uncolorizeHyperrefIfChanges}}{}

\begin{document}
\maketitle

\begin{abstract}
This paper introduces a new notion of a Fenchel conjugate, which generalizes the classical Fenchel conjugation to functions defined on Riemannian manifolds.
We investigate its properties, e.g.,~the Fenchel--Young inequality and the characterization of the convex subdifferential using the analogue of the Fenchel--Moreau Theorem.
These properties of the Fenchel conjugate are employed to derive a Riemannian primal-dual optimization algorithm, and to prove its convergence for the case of Hadamard manifolds under appropriate assumptions.
Numerical results illustrate the performance of the algorithm, which competes with the recently derived Douglas--Rachford algorithm on manifolds of nonpositive curvature.
Furthermore, we show numerically that our novel algorithm even converges on manifolds of positive curvature.\end{abstract}


\begin{AMS}
\href{https://mathscinet.ams.org/msc/msc2010.html?t=49N15}{49N15}, \href{https://mathscinet.ams.org/msc/msc2010.html?t=49M29}{49M29}, \href{https://mathscinet.ams.org/msc/msc2010.html?t=90C26}{90C26}, \href{https://mathscinet.ams.org/msc/msc2010.html?t=49Q99}{49Q99}
\end{AMS}

\input{main.tex}

\printbibliography

\end{document}

%% file: main.tex
\section{Introduction}
\label{sec:Introduction}
Convex analysis plays an important role in optimization, and an elaborate theory on convex analysis and conjugate duality is available on locally convex vector spaces.
Among the vast references on this topic, we mention~\cite{BauschkeCombettes:2011:1} for convex analysis and monotone operator techniques,~\cite{EkelandTemam:1999:1} for convex analysis and the perturbation approach to duality, or~\cite{Rockafellar:1970:1} for an in-depth development of convex analysis on Euclidean spaces.
\cite{Rockafellar:1974:1} focuses on conjugate duality on Euclidean spaces,~\cite{Zalinescu:2002:1,Bot:2010:1} on conjugate duality on locally convex vector spaces, and~\cite{MartinezLegaz:2005:1} on some particular applications of conjugate duality in economics.

We wish to emphasize in particular the role of convex analysis in the analysis and numerical solution of regularized ill-posed problems.
Consider for instance the total variation (TV) functional, which was introduced for imaging applications in the famous Rudin--Osher--Fatemi (ROF) model, see \cite{RudinOsherFatemi:1992:1}, and which is known for its ability to preserve sharp edges.
We refer the reader to~\cite{ChambolleCasellesCremersNovagaPock:2010:1} for further details about total variation for image analysis.
Further applications and regularizers can be found in~\cite{ChambolleLions:1997:1,StrongChan:2003:1,Chambolle:2004:1,ChanEsedogluParkYip:2006:1,WangYangYinZhang:2008:1}.
In addition, higher order differences or differentials can be taken into account, see for example~\cite{ChanMarquinaMulet:2000:1,PapafitsorosSchoenlieb:2014:1} or most prominently the total generalized variation (TGV)~\cite{BrediesKunischPock:2010:1}.
These models use the idea of the pre-dual formulation of the energy functional and Fenchel duality to derive efficient algorithms.
Within the image processing community the resulting algorithms of primal-dual hybrid gradient type are often referred to as the Chambolle--Pock algorithm, see \cite{ChambollePock:2011:1}.

In recent years, optimization on Riemannian manifolds has gained a lot of interest.
Starting in the 1970s, optimization on Riemannian manifolds and corresponding algorithms have been investigated; see for instance \cite{Udriste:1994:1} and the references therein.
In particular, we point out the work by Rapcsák with regard to geodesic convexity in optimization on manifolds; see for instance \cite{Rapcsak:1986:1,Rapcsak:1991:1} and \cite[Ch.~6]{Rapcsak:1997:1}.
The latter reference also serves as a source for optimization problems on manifolds obtained by rephrasing equality constrained problems in vector spaces as unconstrained problems on certain manifolds.
For a com\-pre\-hen\-si\-ve textbook on optimization on matrix manifolds, see~\cite{AbsilMahonySepulchre:2008:1} and the recent \cite{Boumal:2020:1}.

With the emergence of manifold-valued imaging, for example in InSAR imaging~\cite{BuergmannRosenFielding:2000:1}, data consisting of orientations for example in electron back\-scatte\-red diffraction (EBSD)~\cite{AdamsWrightKunze:1993:1,KunzeWrightAdamsDingley:1993:1}, dextrous hand grasping~\cite{DirrHelmkeLageman:2007:1}, or for diffusion tensors in magnetic resonance imaging (DT-MRI), for example discussed in~\cite{PennecFillardAyache:2006:1}, the development of optimization techniques and/or algorithms on manifolds (especially for non-smooth functionals) has gained a lot of attention.
Within these applications, the same tasks appear as for classical, Euclidean imaging, such as denoising, inpainting or segmentation.
Both~\cite{LellmannStrekalovskiyKoetterCremers:2013:1} as well as~\cite{WeinmannDemaretStorath:2014:1} intro\-du\-ced the total variation as a prior in a variational model for manifold-valued images.
While the first extends a lifting approach previously introduced for cyclic data in~\cite{StrekalovskiyCremers:2011:1} to Riemannian manifolds, the latter introduces a cyclic proximal point algorithm (CPPA) to compute a minimizer of the variational model.
Such an algorithm was previously introduced by~\cite{Bacak:2014:1} on $\operatorname{CAT}(0)$ spaces based on the proximal point algorithm introduced by~\cite{FerreiraOliveira:2002:1} on Riemannian manifolds.
Based on these models and algorithms, higher order models have been derived~\cite{BergmannLausSteidlWeinmann:2014:1,BacakBergmannSteidlWeinmann:2016:1,BergmannFitschenPerschSteidl:2018:1,BrediesHollerStorathWeinmann:2018:1}.
Using a relaxation, the half-quadratic minimization \cite{BergmannChanHielscherPerschSteidl:2016:1}, also known as iteratively reweighted least squares (IRLS)~\cite{GrohsSprecher:2016:1}, has been ge\-ne\-ra\-li\-zed to manifold-valued image processing tasks and employs a quasi-Newton method.
Finally, the parallel Douglas--Rachford algorithm (PDRA) was introduced on Hadamard manifolds~\cite{BergmannPerschSteidl:2016:1} and its convergence proof is, to the best of our knowledge, limited to manifolds with constant nonpositive curvature.
Numerically, the PDRA still performs well on arbitrary Hadamard manifolds.
However, for the classical Euclidean case the Douglas--Rachford algorithm is equivalent to applying the alternating directions method of mul\-ti\-pliers (ADMM) \cite{GabayMercier:1976:1} on the dual problem and hence is also equivalent to the algorithm of~\cite{ChambollePock:2011:1}.

In this paper we introduce a new notion of Fenchel duality for Riemannian manifolds, which allows us to derive a conjugate duality theory for convex optimization problems posed on such manifolds.
Our theory allows new al\-go\-rith\-mic approaches to be devised for optimization problems on manifolds.
In the absence of a global concept of convexity on general Riemannian manifolds, our approach is local in nature.
On so-called Hadamard manifolds, however, there is a global notion of convexity and our approach also yields a global method.

The work closest to ours is~\cite{AhmadiKakavandiAmini:2010:1}, who introduce a Fenchel conjugacy-like concept on Hadamard metric spaces, using a quasilinearization map in terms of distances as the duality product.
In contrast, our work makes use of intrinsic tools from differential geometry such as geodesics, tangent and cotangent vectors to establish a conjugation scheme which extends the theory from locally convex vector spaces to Riemannian manifolds.
We investigate the application of the correspondence of a primal problem
\begin{equation}
	\label{eq:primal_problem}
	\text{Minimize} \quad \Fidelity(p) + \Prior(\Lambda(p))
\end{equation}
to a suitably defined dual and derive a primal-dual algorithm on Riemannian manifolds.
In the absence of a concept of linear operators between manifolds we follow the approach of~\cite{Valkonen:2014:1} and state an exact and a linearized variant of our newly established Riemannian Chambolle--Pock algorithm (RCPA).
We then study convergence of the latter on Hadamard manifolds.
Our analysis relies on a careful investigation of the convexity properties of the functions~$F$ and $G$.
We distinguish between geodesic convexity and convexity of a function composed with the exponential map on the tangent space.
Both types of convexity coincide on Euclidean spaces.
This renders the proposed RCPA a direct generalization of the Chambolle-Pock algorithm to Riemannian manifolds.

As an example for a problem of type \eqref{eq:primal_problem}, we detail our algorithm for the anisotropic and isotropic total variation with squared distance data term, \ie, the variants of the ROF model on Riemannian manifolds.
After illustrating the correspondence to the Euclidean (classical) Chambolle--Pock algorithm, we compare the numerical performance of the RCPA to the CPPA and the PDRA.
While the latter has only been shown to converge on Hadamard manifolds of constant curvature, it performs quite well on Hadamard manifolds in general.
On the other hand, the CPPA is known to possibly converge arbitrarily slowly; even in the Euclidean case.
We illustrate that our linearized algorithm competes with the PDRA, and it even performs favorably on manifolds with non-negative curvature, like the sphere.

The remainder of the paper is organized as follows.
In \cref{sec:Preliminaries} we recall a number of classical results from convex analysis in Hilbert spaces.
In an effort to make the paper self-contained, we also briefly state the required concepts from differential geometry.
{\renewcommand*{\sectionautorefname}{Section}\cref{sec:Fenchel_conjugation_scheme}} is devoted to the development of a complete notion of Fenchel conjugation for functions defined on manifolds.
To this end, we extend some classical results from convex analysis and locally convex vector spaces to manifolds, like the Fenchel--Moreau Theorem (also known as the Biconjugation Theorem) and useful characterizations of the subdifferential in terms of the conjugate function.
In \cref{sec:Primal-Dual_on_manifolds} we formulate the primal-dual hybrid gradient method (also referred to as the Riemannian Chambolle--Pock algorithm, RCPA) for general optimization problems on mani\-folds involving non-linear operators.
We present an exact and a linearized formulation of this novel method and prove, under suitable assumptions, convergence for the linearized variant to a minimizer of a linearized problem on arbitrary Hadamard manifolds.
As an application of our theory, \cref{sec:ROF_Models} focuses on the analysis of several total variation models on manifolds.
In \cref{sec:Numerical_experiments} we carry out numerical experiments to illustrate the performance of our novel primal-dual algorithm.
Finally, we give some conclusions and further remarks on future research in \cref{sec:Conclusions}.

\section{Preliminaries on Convex Analysis and Differential Geometry}
\label{sec:Preliminaries}

In this section we review some well known results from convex analysis in Hilbert spaces as well as necessary concepts from differential geometry.
We also revisit the intersection of both topics, convex analysis on Riemannian manifolds, including its subdifferential calculus.

\subsection{Convex Analysis}
\label{subsec:Convex_Analysis}

In this subsection let~$f\colon \cX \to \eR$, where $\eR \coloneqq \R \cup \{\pm\infty\}$ denotes the extended real line and $\cX$ is a Hilbert space with inner product $\inner{\cdot}{\cdot}_\cX$ and duality pairing $\dual{\cdot}{\cdot}_{\cX^*,\cX}$, respectively.
Here, $\cX^*$ denotes the dual space of $\cX$.
When the space $\cX$ and its dual $\cX^*$ are clear from the context, we omit the space and just write $\inner{\cdot}{\cdot}$ and $\dual{\cdot}{\cdot}$, respectively.
For standard definitions like \emph{closedness, properness, lower semicontinuity (lsc)} and \emph{convexity} of $f$ we refer the reader, \eg, to the textbooks~\cite{Rockafellar:1970:1,BauschkeCombettes:2011:1}.

\begin{definition}
	\label{def:Classical_Fenchel_conjugate}%
	The \emph{Fenchel conjugate} of a function~$f\colon \cX \to\eR$ is defined as the function $f^*\colon\cX^* \to \eR$ such that
	\begin{equation}
		\label{eq:Classical_Fenchel_conjugate}
		f^*(x^*)
		\coloneqq
		\sup_{x \in \cX} \paren[auto]\{\}{\dual{x^*}{x} - f(x)}.
	\end{equation}
\end{definition}
We recall some properties of the classical Fenchel conjugate function in the following lemma.

\begin{lemma}[{\cite[Ch.~13]{BauschkeCombettes:2011:1}}]
	\label{lemma:Classical_properties_Fenchel_conjugate}%
	Let~$f,g\colon \cX \to\eR$ be proper functions, $\alpha\in\R$, $\lambda>0$ and $b \in \cX$.
	Then the following statements hold.
	\begin{enumerate}
		\item\label{item:Classical_properties_conjugate}
			$f^*$ is convex and lsc.
		\item\label{item:Classical_properties_inequality}
			If $f(x)\leq g(x)$ for all $x\in \cX $, then $f^*(x^*)\geq g^*(x^*)$ for all $x^*\in \cX^* $.
		\item\label{item:Classical_properties_adding_alpha}
			If $g(x)=f(x)+\alpha$ for all $x\in \cX $, then $g^*(x^*) = f^*(x^*) - \alpha$ for all $x^*\in \cX^* $.
		\item\label{item:Classical_properties_lambda}
			If $g(x) = \lambda f(x)$ for all $x\in \cX $, then $g^*(x^*)=\lambda f^*(x^*/\lambda)$ for all $x^*\in\cX^* $.
		\item\label{item:Classical_properties_Fenchel_shift}
			If $g(x) = f(x+b)$ for all $x\in \cX $, then $g^*(x^*) = f^*(x^*) - \dual{x^*}{b}$ for all $x^*\in \cX^*$.
		\item\label{item:Classical_properties_Fenchel-Young}
			The Fenchel--Young inequality holds, \ie, for all $(x,x^*) \in \cX \times \cX^* $ we have
			\begin{equation}
				\label{eq:Classical_Fenchel_Young_inequality}
				\dual{x^*}{x} \leq f(x) + f^*(x^*).
			\end{equation}
	\end{enumerate}
\end{lemma}

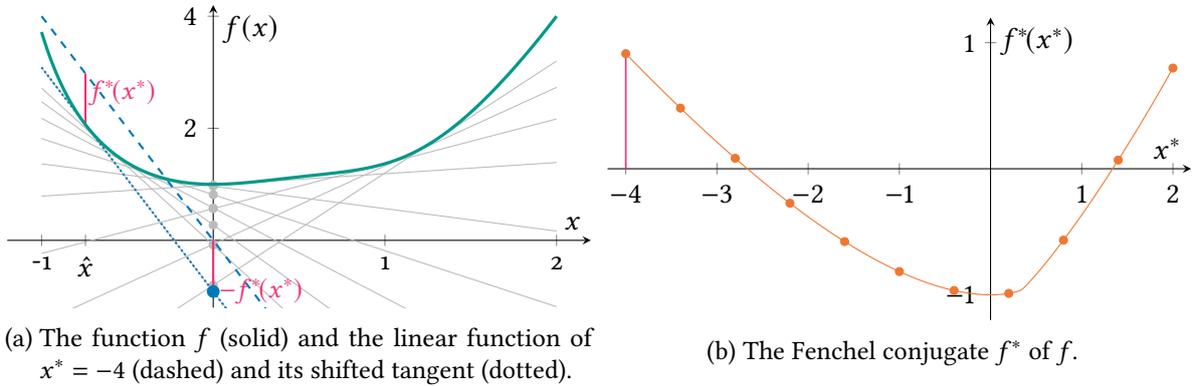
\begin{figure}[htb]\centering
	\pgfplotstableread[col sep = comma]{Data/Rn/FFStarTangents.csv}\FAndTangents%
	\pgfplotstableread[col sep = comma]{Data/Rn/FStarPoints.csv}\FStarPoints%
	\pgfplotstableread[col sep = comma]{Data/Rn/FStarSegments.csv}\FStarSegments%
	\tikzset{fStyle/.style={TolVibrantTeal,very thick}}
	\tikzset{tStyle/.style={TolVibrantGray,mark size=1.5pt, thin}}
	\tikzset{ftStyle/.style={TolVibrantBlue,thick, mark size=2pt}}
	\tikzset{fsStyle/.style={TolVibrantOrange,mark size=1.5pt}}
	\tikzset{sStyle/.style={TolVibrantMagenta,thick}}
	\begin{subfigure}{.49\textwidth}\centering
		\begin{tikzpicture}
			\begin{axis}[
				width=1.2\textwidth,height=.25\textheight,
				xmin=-1.2,xmax=2.2, ymin=-1.2, ymax=4.2,
				clip=true,
				axis y line=center,
				axis x line=center,
				xlabel={$x$},
				ylabel={$f(x)$},
				xtick={-1, -0.7444012767185891, 1, 2},
				xticklabels={-1,$\hat x$,1,2},
				clip mode=individual
				]
				\addplot[tStyle] table[x=x,y=t3]{\FAndTangents};
				\addplot[tStyle,only marks] table[x=z,y=mfs3]{\FStarPoints};
				\addplot[tStyle] table[x=x,y=t4]{\FAndTangents};
				\addplot[tStyle,only marks] table[x=z,y=mfs4]{\FStarPoints};
				\addplot[tStyle] table[x=x,y=t5]{\FAndTangents};
				\addplot[tStyle,only marks] table[x=z,y=mfs5]{\FStarPoints};
				\addplot[tStyle] table[x=x,y=t6]{\FAndTangents};
				\addplot[tStyle,only marks] table[x=z,y=mfs6]{\FStarPoints};
				\addplot[tStyle] table[x=x,y=t7]{\FAndTangents};
				\addplot[tStyle,only marks] table[x=z,y=mfs7]{\FStarPoints};
				\addplot[tStyle] table[x=x,y=t8]{\FAndTangents};
				\addplot[tStyle,only marks] table[x=z,y=mfs8]{\FStarPoints};
				\addplot[tStyle] table[x=x,y=t9]{\FAndTangents};
				\addplot[tStyle,only marks] table[x=z,y=mfs9]{\FStarPoints};
				\addplot[tStyle] table[x=x,y=t10]{\FAndTangents};
				\addplot[tStyle,only marks] table[x=z,y=mfs10]{\FStarPoints};
				\addplot[tStyle] table[x=x,y=t11]{\FAndTangents};
				\addplot[tStyle,only marks] table[x=z,y=mfs11]{\FStarPoints};
				\addplot[ftStyle,dashed] table[x=x,y=t1shift]{\FAndTangents};
				\addplot[sStyle, thick, overlay, postaction = {decorate, decoration={markings, mark=at position .5 with {\node[yshift=2.5, xshift=3ex,font=\small]{${f^*\!(x^*)}$};} }}] table[x=sx,y=sy]{\FStarSegments};
				\addplot[sStyle,thick, overlay, postaction = {decorate, decoration={markings, mark=at position .5 with {\node[yshift=-2ex, xshift=3.85ex,font=\small]{${-f^*\!(x^*)}$};} }}] table[x=tx,y=ty]{\FStarSegments};
				\addplot[tStyle] table[x=x, y=t1]{\FAndTangents};
				\addplot[ftStyle,densely dotted,thick] table[x=x, y=t1]{\FAndTangents};
				\addplot[ftStyle,only marks] table[x=z, y=mfs1]{\FStarPoints};
				\addplot[fStyle] table[x=x, y=f]{\FAndTangents};
			\end{axis}
		\end{tikzpicture}
		\caption{The function $f$ (solid) and the linear function of $x^*=-4$ (dashed) and its shifted tangent (dotted).}
		\label{subfig:realSupportingTangents}
	\end{subfigure}
	\begin{subfigure}{.49\textwidth}\centering
		\begin{tikzpicture}
			\begin{axis}[
				width=1.2\textwidth,height=.25\textheight,
				xmin=-4.2,xmax=2.2, ymin=-1.2, ymax=1.2,
				clip=false,
				axis y line=center,
				axis x line=center,
				xlabel={$x^*$},
				ylabel={$f^*\!(x^*)$}
				]
				\addplot[sStyle,overlay,only marks,ybar,bar width=.25pt,line width=.25pt] table[x=x1,y=fs1]{\FStarPoints};
				\addplot[fsStyle,only marks] table[x=x1,y=fs1]{\FStarPoints};
				\addplot[fsStyle,only marks] table[x=x2,y=fs2]{\FStarPoints};
				\addplot[fsStyle,only marks] table[x=x3,y=fs3]{\FStarPoints};
				\addplot[fsStyle,only marks] table[x=x4,y=fs4]{\FStarPoints};
				\addplot[fsStyle,only marks] table[x=x5,y=fs5]{\FStarPoints};
				\addplot[fsStyle,only marks] table[x=x6,y=fs6]{\FStarPoints};
				\addplot[fsStyle,only marks] table[x=x7,y=fs7]{\FStarPoints};
				\addplot[fsStyle,only marks] table[x=x8,y=fs8]{\FStarPoints};
				\addplot[fsStyle,only marks] table[x=x9,y=fs9]{\FStarPoints};
				\addplot[fsStyle,only marks] table[x=x10,y=fs10]{\FStarPoints};
				\addplot[fsStyle,only marks] table[x=x11,y=fs11]{\FStarPoints};
				\addplot[fsStyle] table[x=xi,y=fstar]{\FAndTangents};
			\end{axis}
		\end{tikzpicture}
		\caption{The Fenchel conjugate $f^*$ of $f$.}
		\label{subfig:realFenchelDual}
	\end{subfigure}
	\caption{Illustration of the Fenchel conjugate for the case $d = 1$ as an interpretation by
	the tangents of slope $x^*$.}
\end{figure}
The Fenchel conjugate of a function $f\colon\R^d\to\eR$ can be interpreted as a maximum seeking problem on the \emph{epigraph} $\epi f \coloneqq \setDef{(x,\alpha) \in \R^d\times \R}{f(x) \leq \alpha}$.
For the case $d=1$ and some fixed~$x^*$ the conjugate maximizes the (signed) distance $\dual{x^*}{x} -f(x)$ of the line of slope~$x^*$ to~$f$.
For instance, let us focus on the case $x^*=-4$ highlighted in \cref{subfig:realSupportingTangents}.
For the linear functional $g_{x^*}(x) = \dual{x^*}{x}$ (dashed), the maximal distance is attained at~$\hat x$.
We can find the same value by considering the shifted functional $h_{x^*}(x) = g_{x^*}(x) - f^*(x^*)$ (dotted line) and its negative value at the origin, \ie, $-h_{x^*}(0) = f^*(x^*)$.
Furthermore $h_{x^*}$ is actually tangent to $f$ at the aforementioned maximizer $\hat x$.
The function $h_{x^*}$ also illustrates the shifting property from~\cref{lemma:Classical_properties_Fenchel_conjugate}~\cref{item:Classical_properties_Fenchel_shift} and its linear offset $ - \dual{x^*}{b}$.
The overall plot of the Fenchel conjugate $f^*$ over an interval of values~$x^*$ is shown in \cref{subfig:realFenchelDual}.

We now recall some results related to the definition of the subdifferential of a proper function.

\begin{definition}[{\cite[Def.~16.1]{BauschkeCombettes:2011:1}}]
	\label{def:Classiscal_convex_subdifferential}%
	Let $f\colon \cX \to\eR$ be a proper function.
	Its subdifferential is defined as
	\begin{equation}
		\partial f(x)
		\coloneqq
		\setDef[auto]{x^* \in \cX^*}{f(z) \geq f(x) + \dual{x^*}{z - x} \text { for all }z \in \cX}.
	\end{equation}
\end{definition}

\begin{theorem}[{\cite[Prop.~16.9]{BauschkeCombettes:2011:1}}]
	\label{thm:Classical_characterization_subdifferential}%
	Let $f\colon \cX \to \eR$ be a proper function and~$x \in \cX$.
	Then $x^*\in\partial f(x)$ holds if and only if
	\begin{equation}
		f(x)+f^*(x^*) = \dual{x^*}{x}.
	\end{equation}
\end{theorem}

\begin{corollary}[{\cite[Thm.~16.23]{BauschkeCombettes:2011:1}}]
	\label{cor:Classical_symmetry_subdifferential}%
	Let~$f \colon \cX \to \eR$ be a lsc, proper, and convex function and~$x^* \in \cX^*$.
	Then $x \in \partial f^*(x^*)$ holds if and only if $x^* \in \partial f(x)$.
\end{corollary}

The Fenchel biconjugate~$f^{**} \colon \cX \to\eR$ of a function~$f \colon \cX \to \eR$ is given by
\begin{equation}
	\label{eq:Classical_biconjugate_function}
	f^{**}(x)
	=
	(f^*)^*(x)
	=
	\sup_{x^*\in \cX^*} \paren[auto]\{\}{\dual{x^*}{x} -f^*(x^*)}.
\end{equation}

Finally, we conclude this section with the following result known as the Fenchel--Moreau or Biconjugation Theorem.

\begin{theorem}[{\cite[Thm.~13.32]{BauschkeCombettes:2011:1}}]
	\label{thm:Classical_Fenchel_Moreaou_theorem}%
	Given a proper function~$f \colon \cX \to \eR$, the equality~$f^{**}(x) = f(x)$ holds for all $x \in \cX$ if and only if $f$ is lsc and convex.
	In this case $f^*$ is proper as well.
\end{theorem}

\subsection{Differential Geometry}
\label{subsec:Differential_Geometry}

This section is devoted to the collection of necessary concepts from differential geometry.
For details concerning the subsequent definitions, the reader may wish to consult~\cite{DoCarmo:1992:1,Lee:2003:1,Jost:2017:1}.

Suppose that~$\cM$ is a~$d$-dimensional connected, smooth manifold.
The tangent space at $p \in \cM$ is a vector space of dimension~$d$ and it is denoted by $\tangent{p}$.
Elements of $\tangent{p}$, \ie, \emph{tangent vectors}, will be denoted by $X_p$ and $Y_p$ etc.\ or simply $X$ and $Y$ when the base point is clear from the context.
The disjoint union of all tangent spaces, \ie,
\begin{equation}
	\label{eq:Tangent_bundle}
	\tangentBundle \coloneqq \bigcup_{p\in\cM}\tangent{p},
\end{equation}
is called the \emph{tangent bundle} of~$\cM$.
It is a smooth manifold of dimension~$2d$.

The dual space of~$\tangent{p}$ is denoted by~$\cotangent{p}$ and it is called the \emph{cotangent space} to $\cM$ at $p$.
The disjoint union
\begin{equation}
	\cotangentBundle \coloneqq \bigcup_{p\in\cM} \cotangent{p}
\end{equation}
is known as the \emph{cotangent bundle}.
Elements of~$\cotangent{p}$ are called \emph{cotangent vectors} to~$\cM$ at~$p$ and they will be denoted by $\xi_p$ and $\eta_p$ or simply $\xi$ and $\eta$.
The natural duality product between $X \in \tangent{p}$ and $\xi \in \cotangent{p}$ is denoted by $\dual{\xi}{X}=\xi(X) \in\R$.

We suppose that $\cM$ is equipped with a Riemannian metric, \ie, a smoothly varying family of inner products on the tangent spaces $\tangent{p}$.
The metric at $p \in \cM$ is denoted by $\riemannian{\cdot}{\cdot}[p] \colon \tangent{p}\times \tangent{p} \to\R$.
The induced norm on~$\tangent{p}$ is denoted by~$\riemanniannorm{\cdot}[p]$.
The Riemannian metric furnishes a linear bijective correspondence between the tangent and cotangent spaces via the Riesz map and its inverse, the so-called \emph{musical isomorphisms}; see~\cite[Ch.~8]{Lee:2003:1}.
They are defined as
\begin{equation}
	\label{eq:Flat_isomorphism}
	\flat\colon\tangent{p}\ni X \mapsto X^\flat\in\cotangent{p}
\end{equation}
satisfying
\begin{equation}
	\label{eq:How_flat_acts}
	\dual{X^\flat}{Y} = \riemannian{X}{Y}[p] \text{ for all } Y\in\tangent{p},
\end{equation}
and its inverse,
\begin{equation}
	\label{eq:Sharp_isomorphism}
	\sharp\colon\cotangent{p}\ni\xi\mapsto\xi^\sharp\in\tangent{p}
\end{equation}
satisfying
\begin{equation}
	\label{eq:How_sharp_acts}
	\riemannian{\xi^\sharp}{Y}[p] = \dual{\xi}{Y} \text{ for all } Y\in\tangent{p}.
\end{equation}
The $\sharp$-isomorphism further introduces an inner product and an associated norm on the cotangent space $\cotangent{p}$, which we will also denote by~$\riemannian{\cdot}{\cdot}[p]$ and $\riemanniannorm{\cdot}[p]$, since it is clear which inner product or norm we refer to based on the respective arguments.

The tangent vector of a curve~$c \colon I \to \cM$ defined on some open interval~$I$ is denoted by $\dot c(t)$.
A curve is said to be geodesic if the directional (covariant) derivative of its tangent in the direction of the tangent vanishes, \ie, if $\nabla_{\dot c(t)} \dot c(t) = 0$ holds for all $t \in I$, where $\nabla$ denotes the Levi-Cevita connection, cf.~\cite[Ch.~2]{DoCarmo:1992:1} or \cite[Thm.~4.24]{Lee:2018:1}.
As a consequence, geodesic curves have constant speed.

We say that a geodesic connects $p$ to $q$ if $c(0)=p$ and $c(1)=q$ holds.
Notice that a geodesic connecting $p$ to $q$ need not always exist, and if it exists, it need not be unique.
If a geodesic connecting $p$ to $q$ exists, there also exists a shortest geodesic among them, which may in turn not be unique.
If it is, we denote the unique shortest geodesic connecting $p$ and $q$ by $\geodesic<a>{p}{q}$.

Using the length of piecewise smooth curves, one can introduce a notion of metric (also known as Riemannian distance) $\dist{\cdot}{\cdot}$ on $\cM$; see for instance \cite[Ch.~2, pp.33--39]{Lee:2018:1}.
As usual, we denote by
\begin{equation}
	\cB_r(p) \coloneqq \setDef{y\in \cM}{\dist{p}{q} < r}
\end{equation}
the open metric ball of radius $r > 0$ with center $p \in \cM$.
Moreover, we define $\cB_\infty(p) = \bigcup_{r > 0} \cB_r(p)$.

We denote by~$\geodesic{p}{X} \colon I \to \cM$, with $I\subset\R$ being an open interval containing~$0$, a geodesic starting at~$p$ with~$\dot{\gamma}_{p,X}(0) = X$ for some $X\in \tangent{p}$.
We denote the subset of $\tangent{p}$ for which these geodesics are well defined until $t=1$ by $\cG_p$.
A Riemannian manifold~$\cM$ is said to be \emph{complete} if $\cG_p = \tangent{p}$ holds for some, and equivalently for all~$p \in \cM$.

The \emph{exponential map} is defined as the function~$\exponential{p} \colon \cG_p\to\cM$ with~$\exponential{p} X \coloneqq \geodesic{p}{X}(1)$.
Note that~$\exponential{p}(tX) = \geodesic{p}{X}(t)$ holds for every $t\in [0,1]$.
We further introduce the set $\cG'_p\subset\tangent{p}$ as some open ball of radius $0 < r \le \infty$ about the origin such that $\exponential{p} \colon \cG'_p \to \exponential{p}(\cG'_p)$ is a diffeomorphism.
The \emph{logarithmic map} is defined as the inverse of the exponential map, \ie,~$\logarithm{p} \colon \exponential{p}(\cG'_p) \to \cG'_p \subset \tangent{p}$.

In the particular case where the sectional curvature of the manifold is nonpositive everywhere, all geodesics connecting any two distinct points are unique.
If furthermore, the manifold
is simply connected and complete, the manifold is called a \emph{Hadamard manifold}, see~\cite[p.10]{Bacak:2014:2}.
Then the exponential and logarithmic maps are defined globally.

Given~$p,q\in \cM$ and~$X\in\tangent{p}$, we denote by~$\parallelTransport{p}{q}{X}$ the so-called \emph{parallel transport} of $X$ along a unique shortest geodesic~$\geodesic<a>{p}{q}$.
Using the musical isomorphisms presented above, we also have a parallel transport of cotangent vectors along geodesics according to
\begin{equation}
	\label{eq:Definition_PT_for_covectors}
	\parallelTransport{p}{q}{\xi_p} \coloneqq \paren[big](){\parallelTransport{p}{q}{\xi_p^\sharp}}^\flat.
\end{equation}

Finally, by a Euclidean space we mean $\R^d$ (where $\tangent{p}[\R^d] = \R^d$ holds), equipped with the Riemannian metric given by the Euclidean inner product.
In this case, $\exponential{p} X = p + X$ and $\logarithm{p} q = q - p$ hold.

\subsection{Convex Analysis on Riemannian Manifolds}
\label{subsec:Convex_analysis_on_Riemannian_manifolds}

Throughout this subsection, $\cM$ is assumed to be a complete and connected Riemannian manifold and we are going to recall the basic concepts of convex analysis on $\cM$.
The central idea is to replace straight lines in the definition of convex sets in Euclidean vector spaces by geodesics.

\begin{definition}[{\cite[Def.~IV.5.1]{Sakai:1996:1}}]
	\label{def:Geodesically_convex_set}%
	A subset~$\cC\subset \cM$ of a Riemannian manifold~$\cM$ is said to be \emph{strongly convex} if for any two points $p, q \in \cC$, there exists a unique shortest geodesic of $\cM$ connecting~$p$ to $q$, and that geodesic, denoted by $\geodesic<a>{p}{q}$, lies completely in $\cC$.
\end{definition}
On non-Hadamard manifolds, the notion of strongly convex subsets can be quite restrictive.
For instance, on the round sphere~$\S^{n}$ with $n \ge 1$, a metric ball $\cB_r(p)$ is strongly convex if and only if $r < \pi/2$.

\begin{definition}
	\label{def:tangent_subset}%
	Let $\cC\subset\cM$ and $p\in\cC$.
	We introduce the tangent subset $\LocalExponentialPreimage{p}{\cC} \subset\tangent{p}$ as
	\begin{equation*}
		\LocalExponentialPreimage{p}{\cC}
		\coloneqq
		\setDef[auto]{X \in \tangent{p}}{\exponential{p} X \in \cC \text{ and } \riemanniannorm{X}[p] = \dist[big]{\exponential{p} X}{p}},
	\end{equation*}
	a localized variant of the pre-image of the exponential map.
\end{definition}

Note that if~$\cC$ is strongly convex, the exponential and logarithmic maps introduce bijections between~$\cC$ and~$\LocalExponentialPreimage{p}{\cC}$ for any~$p\in\cC$.
In particular, on a Hadamard manifold~$\cM$, we have $\LocalExponentialPreimage{p}{\cM} = \tangent{p}$.

The following definition states the important concept of convex functions on Riemannian manifolds.
\begin{definition}[{\cite[Def.~IV.5.9]{Sakai:1996:1}}] \hfill
	\label{def:F_geodesically_convex}%
	\begin{enumerate}
		\item
			A function~$F\colon \cM\to\eR$ is \emph{proper} if $\dom F \coloneqq \setDef{p\in \cM}{F(p)<\infty} \neq \emptyset$ and $F(p)>-\infty$ holds for all~$p\in\cM$.
		\item
			\label{item:F_geodesically_convex}
			Suppose that $\cC \subset \cM$ is strongly convex.
			A function~$F\colon\cM\to \eR$ is called geodesically convex on~$\cC\subset \cM$ if, for all~$p,q\in \cC$, the composition $F \circ \geodesic<a>{p}{q}$ is a convex function on $[0,1]$ in the classical sense.
			Similarly, $F$ is called strictly or strongly convex if $F \circ \geodesic<a>{p}{q}$ fulfills these properties.
		\item
			\label{item:epigraph}
			Suppose that $A \subset \cM$.
			The \emph{epigraph} of a function $F \colon A \to \eR$ is defined as
			\begin{equation}
				\label{eq:Epigraph}
				\epi F
				\coloneqq
				\setDef{(p,\alpha)\in A \times \R}{F(p) \leq \alpha}.
			\end{equation}
		\item
			\label{item:lsc}
			Suppose that $A \subset \cM$.
			A proper function~$F\colon A\to\eR$ is called \emph{lower semicontinuous (lsc)} if $\epi F$ is closed.
	\end{enumerate}
\end{definition}
Suppose that $\cC \subset \cM$ is strongly convex and $F \colon \cC \to \eR$, then an equivalent way to describe its lower semicontinuity (\Cref{item:lsc}) is to require that the composition
\begin{equation}\label{eq:F_composed_with_exp_m}
	F \circ \exponential{m} \colon \LocalExponentialPreimage{m}{\cC} \to \eR
\end{equation}
is lsc for an arbitrary $m \in \cC$ in the classical sense, where $\LocalExponentialPreimage{m}{\cC}$ is defined in \cref{def:tangent_subset}.

We now recall the notion of the subdifferential of a geodesically convex function defined on a Riemannian manifold.
\begin{definition}[{\cite{FerreiraOliveira:1998:1}, \cite[Def.~3.4.4]{Udriste:1994:1}}]
	\label{def:Subdifferential}%
	Suppose that $\cC \subset \cM$ is strongly convex.
	The \emph{subdifferential}~$\partial_\cM F$ on~$\cC$ at a point~$p\in \cC$ of a proper, geodesically convex function~$F\colon\cC\to\eR$ is given by
	\begin{equation}\label{eq:Subdifferential_manifold}
		\partial_\cM F(p)
		\coloneqq
		\setDef[auto]{\xi \in \cotangent{p}}{F(q) \geq F(p) + \dual{\xi}{\logarithm{p} q} \text{ for all } q \in \cC}
		.
	\end{equation}
\end{definition}
In the above notation, the index~$\cM$ refers to the fact that it is the Riemannian subdifferential; the set~$\cC$ should always be clear from the context.

We further recall the definition of the proximal map, which was generalized to Hadamard manifolds in~\cite{FerreiraOliveira:2002:1}.
\begin{definition}
	\label{def:prox}%
	Let~$\cM$ be a Riemannian manifold,~$F\colon \cM \rightarrow \eR$ be proper, and~$\lambda > 0$.
	The \emph{proximal map} of~$F$ is defined as
	\begin{equation}
		\label{eq:proximal_map}
		\proxOp_{\lambda \, F }(p)
		\coloneqq
		\Argmin_{q \in \cM}
		\paren[auto]\{\}{\frac{1}{2} \dist{p}{q}[2] + \lambda \, F(q)}
		.
	\end{equation}
\end{definition}
Note that on Hadamard manifolds, the proximal map is single-valued for proper geodesically convex functions; see~\cite[Ch.~2.2]{Bacak:2014:2} or~\cite[Lem.~4.2]{FerreiraOliveira:2002:1} for details.
The following lemma is used later on to characterize the proximal map using the subdifferential on Hadamard manifolds.

\begin{lemma}[{\cite[Lem.~4.2]{FerreiraOliveira:2002:1}}]
	\label{lem:proxSubdiff}%
	Let~$F \colon \cM \to \eR$ be a proper, geodesically convex function on the Hadamard manifold~$\cM$.
	Then the equality $q = \prox{\lambda \, F}{p}$ is equivalent to
	\begin{equation}
		\label{eq:proxSubdiff}
		\frac{1}{\lambda}\paren[big](){\logarithm{q} p}^{\flat} \in \partial_\cM F(q).
	\end{equation}
\end{lemma}

\section{Fenchel Conjugation Scheme on Manifolds}
\label{sec:Fenchel_conjugation_scheme}

In this section we present a novel Fenchel conjugation scheme for extended real-valued functions defined on manifolds.
We generalize ideas from~\cite{Bertsekas:1978:1}, who defined local conjugation on manifolds embedded in $\R^d$ specified by nonlinear equality constraints.

Throughout this section, suppose that $\cM$ is a Riemannian manifold and $\cC \subset \cM$ is strongly convex.
The definition of the Fenchel conjugate of~$F$ is motivated by \cite[Thm.~12.1]{Rockafellar:1970:1}.

\begin{definition}
	\label{def:Fenchel_conjugate_on_M}%
	Suppose that $F\colon\cC \to\eR$, where $\cC\subset\cM$ is strongly convex, and $m \in \cC$.
	The \emph{$m$-Fenchel conjugate} of $F$ is defined as the function $F_m^*\colon \cotangent{m} \to \eR$ such that
	\begin{equation}
		\label{eq:Fenchel_conjugate_on_M}
		F_m^*(\xi_m)
		\coloneqq
		\sup_{X \in\LocalExponentialPreimage{m}{\cC}}
		\paren[auto]\{\}{\dual{\xi_m}{X} - F(\exponential{m} X)}
		,
		\quad
		\xi_m \in \cotangent{m}
		.
	\end{equation}
\end{definition}

\begin{remark}
	\label{rem:dependency_base}%
	Note that the Fenchel conjugate~$F_m^*$ depends on both the strongly convex set~$\cC$ and on the base point~$m$.
	Observe as well that when~$\cM$ is a Hadamard manifold, it is possible to have $\cC = \cM$.
	In the particular case of the Euclidean space $\cC = \cM = \R^d$, \cref{def:Fenchel_conjugate_on_M} becomes
	\makeatletter
	\begin{equation*}
		\begin{aligned}
			F_m^* (\xi)
			&
			=
			\sup_{X \in\R^d}
			\paren[auto]\{\}{\dual{\xi}{X} - F(m+X)}
			=
			\sup_{Y \in\R^d} \paren[auto]\{\}{\dual{\xi}{Y-m} - F(Y)}
			\IfSubStr{\@currentclass}{svjour3.cls}{%
				\\
				&
				}{%
			}
			=
			F^*(\xi) - \dual{\xi}{m}
		\end{aligned}
	\end{equation*}
	\makeatother
	for $\xi \in \R^d$.
	Hence, taking~$m$ to be the zero vector we recover the classical (Euclidean) conjugate~$F^*$ from~\cref{def:Classical_Fenchel_conjugate} with $\cX = \R^n$.
\end{remark}

\begin{example}
	\label{example:distance_squared}
	Let $\cM$ be a Hadamard manifold, $m\in\cM$ and $F\colon\cM\to\R$ defined as $F(p) = \frac{1}{2} \dist{p}{m}[2]$.
	Due to the fact that
	\begin{equation*}
		F(p)
		=
		\frac{1}{2} \dist{p}{m}[2]
		=
		\frac{1}{2} \riemanniannorm{\logarithm{m} p}[m]^2,
	\end{equation*}
	we obtain from \cref{def:Fenchel_conjugate_on_M} the following representation of the $m$-conjugate of~$F$:
	\begin{align*}
		F_m^*(\xi_m)
		&
		=
		\sup_{X \in\tangent{m}} \paren[Big]\{\}{\dual{\xi_m}{X} - \frac{1}{2}\riemanniannorm{\logarithm{m} \exponential{m} X}[m]^2}
		\\
		&
		=
		\sup_{X \in\tangent{m}} \paren[Big]\{\}{\dual{\xi_m}{X} - \frac{1}{2}\riemanniannorm{X}[m]^2}
		=
		\frac{1}{2}\riemanniannorm{\xi_m}[m]^2
		.
	\end{align*}
	Notice that the conjugate w.r.t.\ base points other than~$m$ does not have a similarly simple expression.
	In the Euclidean setting with $\cM = \R^d$ and $F(p) = \frac{1}{2} \norm{p-m}^2$, it is well known that
	\begin{equation*}
		F_0^*(\xi)
		=
		F^*(\xi)
		=
		\frac{1}{2} \norm{\xi + m}^2 - \frac{1}{2} \norm{m}^2
	\end{equation*}
	holds and thus, by \cref{rem:dependency_base},
	\begin{equation*}
		F_m^*(\xi)
		=
		F^*(\xi)
		-
		\dual{\xi}{m}
		=
		\frac{1}{2} \norm{\xi}^2
	\end{equation*}
	holds in accordance with the expression obtained above.
\end{example}

We now establish a result regarding the properness of the~$m$-conjugate function, generalizing a result from \cite[Prop.~13.9]{BauschkeCombettes:2011:1}.
\begin{lemma}
	\label{lemma:F*_proper_implies_F_proper}%
	Suppose that $F\colon\cC \to\eR$ and $m \in \cC$ where $\cC$ is strongly convex.
	If~$F_m^*$ is proper, then~$F$ is also proper.
\end{lemma}
\begin{proof}
	Since $F_m^*$ is proper we can pick some $\xi_m\in\dom F_m^*$.
	Hence, applying \cref{def:Fenchel_conjugate_on_M} we get
	\begin{equation*}
		F_m^*(\xi_m) =
		\sup_{X\in\LocalExponentialPreimage{m}{\cC}}
		\paren[auto]\{\}{\dual{\xi_m}{X} - F(\exponential{m} X)} < +\infty,
	\end{equation*}
	so there must exist at least one~$\bar X \in \LocalExponentialPreimage{m}{\cC}$ such that~$F(\exponential{m} \bar X) \in \R$.
	This shows that $F \not \equiv + \infty$.
	On the other hand, let $p \in \cC$ and take $X \coloneqq \logarithm{m} p$.
	If $F(p)$ were equal to $- \infty$, then $F_m^* (\xi_m) = +\infty$ for any $\xi_m \in \cotangent{m}$, which would contradict the properness of $F_m^*$.
	Consequently, $F$ is proper.
\end{proof}

\begin{definition}
	\label{definition:biconjugate}%
	Suppose that $F\colon\cC \to\eR$, where $\cC$ is strongly convex, and $m,m' \in \cC$.
	Then the \emph{($mm^\prime$)-Fenchel biconjugate function} $F_{mm'}^{**} \colon \cC\to\R$ is defined as
	\begin{equation}
		\label{eq:Manifold_Biconjugate_function_original}
		F_{mm'}^{**}(p)
		=
		\sup_{\xi_{m^\prime} \in \cotangent{m'}}
		\paren[auto]\{\}{\dual{\xi_{m^\prime}}{\logarithm{m'} p} - F_m^*(\parallelTransport{m^\prime}{m}{\xi_{m^\prime}})}
		,
		\quad p \in \cC
		.
	\end{equation}
\end{definition}
Note that $F_{mm'}^{**}$ is again a function defined on the Riemannian manifold.
The relation between $F_{mm}^{**}$ and $F$ is discussed further below, as well as properties of higher order conjugates.

\begin{lemma}
	\label{lemma:Manifold_inequality_biconjugate}%
	Suppose that $F\colon\cC \to\eR$ and $m \in \cC$.
	Then $F_{m m}^{**}(p) \leq F(p)$ holds for all~$p\in \cC$.
\end{lemma}
\begin{proof}
	Applying~\eqref{eq:Manifold_Biconjugate_function_original}, we have
	\begin{align*}
		F_{m m}^{**}(p)
		&
		=
		\sup_{\xi_m\in\cotangent{m}} \paren[auto]\{\}{\dual{\xi_m}{\logarithm{m} p} - F_m^*(\xi_m)}
		\\
		&
		=
		\sup_{\xi_m\in\cotangent{m}} \paren[Big]\{\}{\dual{\xi_m}{\logarithm{m} p} - \sup_{X\in\LocalExponentialPreimage{m}{\cC}} \paren[auto]\{\}{\dual{\xi_m}{X} - F(\exponential{m} X)}}
		\\
		&
		=
		\sup_{\xi_m\in\cotangent{m}} \paren[Big]\{\}{\dual{\xi_m}{\logarithm{m} p} + \inf_{X\in\LocalExponentialPreimage{m}{\cC}} \paren[auto]\{\}{-\dual{\xi_m}{X} + F(\exponential{m} X)}}
		\\
		&
		\leq
		\sup_{\xi_m\in\cotangent{m}} \paren[auto]\{\}{\dual{\xi_m}{\logarithm{m} p} - \dual{\xi_m}{\logarithm{m} p} + F \paren[big](){\exponential{m} \logarithm{m} p}}
		\\
		&
		=
		F(p)
		,
	\end{align*}
	which finishes the proof.
\end{proof}

The following lemma proves that our definition of the Fenchel conjugate enjoys properties~\cref{item:Classical_properties_inequality}--\cref{item:Classical_properties_lambda} stated in \cref{lemma:Classical_properties_Fenchel_conjugate} for the classical definition of the conjugate on a Hilbert space.
Results parallel to properties~\cref{item:Classical_properties_conjugate} and \cref{item:Classical_properties_Fenchel-Young} in \cref{lemma:Classical_properties_Fenchel_conjugate} will be given in \cref{lem:pro.Fm} and \cref{prop:Manifold_Fenchel_Young_inequality}, respectively.
Observe that an analogue of property~\cref{item:Classical_properties_Fenchel_shift} in \cref{lemma:Classical_properties_Fenchel_conjugate} cannot be expected for $F\colon \cM\to\R$ due to the lack of a concept of linearity on manifolds.

\begin{lemma}
	\label{lemma:Manifold_conjugate_properties}%
	Suppose that $\cC \subset \cM$ is strongly convex.
	Let $F, G \colon \cC \to \eR$ be proper functions, $m \in \cC$, $\alpha \in \R$ and $\lambda > 0$.
	Then the following statements hold.
	\begin{enumerate}
		\item
			\label{item:FDlessThan}
			If $F(p)\leq G(p)$ for all $p\in \cC$, then $F_m^*(\xi_m) \geq G_m^*(\xi_m)$ for all $\xi_m \in \cotangent{m}$.
		\item
			\label{item:FDplus}
			If $G(p) = F(p)+\alpha$ for all $p\in \cC$, then $G_m^*(\xi_m) = F_m^*(\xi_m)-\alpha$ for all $\xi_m\in\cotangent{m}$.
		\item
			\label{item:FDscale}
			If $G(p) = \lambda \, F(p)$ for all $p\in \cC$, then $G_{m}^*(\xi_m) = \lambda \, F_m^* \paren[big](){\frac{\xi_m}{\lambda}}$ for all~$\xi_m\in\cotangent{m}$.
	\end{enumerate}
\end{lemma}
\begin{proof}
	If $F(p) \leq G(p)$ for all~$p\in \cC$, then it also holds~$F(\exponential{m} X) \leq G(\exponential{m} X)$ for every~$X\in\LocalExponentialPreimage{m}{\cC}$.
	Then we have for any~$\xi_{m}\in\cotangent{m}$ that
	\begin{align*}
		F_m^*(\xi_m)
		&
		=
		\sup_{X \in \LocalExponentialPreimage{m}{\cC}} \paren[auto]\{\}{\dual{\xi_m}{X} - F(\exponential{m} X)}
		\\
		&
		\geq
		\sup_{X\in\LocalExponentialPreimage{m}{\cC}} \paren[auto]\{\}{\dual{\xi_m}{X} - G(\exponential{m} X)}
		=
		G_m^*(\xi_m)
		.
	\end{align*}
	This shows \cref{item:FDlessThan}.
	Similarly, we prove~\cref{item:FDplus}:
	let us suppose that~$G(p) = F(p)+\alpha$ for all~$p\in \cC$.
	Then $G(\exponential{m} X) = F(\exponential{m} X) + \alpha$ for every~$X\in\LocalExponentialPreimage{m}{\cC}$.
	Hence, for any~$\xi_{m}\in\cotangent{m}$ we obtain
	\begin{align*}
		G_m^*(\xi_m)
		&
		=
		\sup_{X\in\LocalExponentialPreimage{m}{\cC}} \paren[auto]\{\}{\dual{\xi_m}{X} - G(\exponential{m} X)}
		\\
		&
		=
		\sup_{X\in\LocalExponentialPreimage{m}{\cC}} \paren[auto]\{\}{\dual{\xi_m}{X} - (F(\exponential{m} X) + \alpha)}
		\\
		&
		=
		\sup_{X\in\LocalExponentialPreimage{m}{\cC}} \paren[audo]\{\}{\dual{\xi_m}{X} - F(\exponential{m} X)} - \alpha
		=
		F_m^*(\xi_m) -\alpha
		.
	\end{align*}
	Let us now prove~\cref{item:FDscale} and suppose that~$\lambda >0$ and~$G(\exponential{m} X) = \lambda \, F(\exponential{m} X)$ for all~$X\in\LocalExponentialPreimage{m}{\cC}$.
	Then we have for any~$\xi_{m}\in\cotangent{m}$ that
	\begin{align*}
		G_m^*(\xi_m)
		&
		=
		\sup_{X\in\LocalExponentialPreimage{m}{\cC}} \paren\{\}{\dual{\xi_m}{X} - G(\exponential{m} X)}
		\\
		&
		=
		\sup_{X\in\LocalExponentialPreimage{m}{\cC}} \paren[auto]\{\}{\dual{\xi_m}{X} - \lambda \, F(\exponential{m} X)}
		\\
		&
		=
		\lambda \sup_{X\in\LocalExponentialPreimage{m}{\cC}} \paren[auto]\{\}{\dual{\lambda^{-1}{\xi_m}}{X} - F(\exponential{m} X)}
		=
		\lambda \, F_m^*\paren[big](){\tfrac{\xi_m}{\lambda}}
		.
	\end{align*}
\end{proof}

Suppose that $F\colon\cC \to\eR$, where $\cC$ is strongly convex, and $m,m',m'' \in \cC$.
The following proposition addresses the triconjugate~$F_{mm'm''}^{***} \colon \cotangent{m''} \to \eR$ of~$F$, which we define as
\begin{equation}
	\label{eq:triconjugate}
	F_{mm'm''}^{***}
	\coloneqq
	(F_{mm'}^{**})_{m''}^*
	.
\end{equation}
\begin{proposition}
	\label{proposition:triconjugate}
	Suppose that $\cM$ is a Hadamard manifold, $m \in \cM$ and $F \colon \cM \to \eR$.
	Then the following holds:
	\begin{equation}\label{eq:m_Triconjugate}
		F_{mmm}^{***} = (F_{mm}^{**})_{m}^* = (F_m^*)^{**} = F_m^*
		\quad \text{on } \cotangent{m}.
	\end{equation}
\end{proposition}
\begin{proof}
	Using \cref{def:Classical_Fenchel_conjugate,def:Fenchel_conjugate_on_M,definition:biconjugate}, it is easy to see that
	\begin{equation*}
		(F_m^*)^* (\logarithm{m} p)
		=
		F_{mm}^{**}(p)
	\end{equation*}
	holds for all $p$ in $\cM$.
	Now \eqref{eq:triconjugate}, \cref{def:Fenchel_conjugate_on_M}, and the bijectivity of $\exponential{m}$ and $\logarithm{m}$ imply that
	\begin{align*}
		F_{mmm}^{***}(\xi_m)
		=
		(F_{mm}^{**})_{m}^*
		&
		=
		\sup_{X \in \tangent{m}} \paren[auto]\{\}{\dual{\xi_m}{X} - F_{mm}^{**}(\exponential{m} X)}
		\\
		&
		=
		\sup_{p \in \cM} \paren[auto]\{\}{\dual{\xi_m}{\logarithm{m} p} - F_{mm}^{**}(p)}
		\\
		&
		=
		\sup_{p \in \cM} \paren[auto]\{\}{\dual{\xi_m}{\logarithm{m} p} - (F_m^*)^*(\logarithm{m} p)}
	\end{align*}
	holds for all $\xi_m\in\cotangent{m}$.
	We now set $f_m \coloneqq F \circ \exponential{m}$ and use \cref{def:Classical_Fenchel_conjugate,def:Fenchel_conjugate_on_M} to infer that
	\begin{equation*}
		F_m^*(\xi_m)
		=
		\sup_{X \in \tangent{m}} \paren[auto]\{\}{\dual{\xi_m}{X} - F(\exponential{m} X)}
		=
		f_m^*(\xi_m)
	\end{equation*}
	holds for all $\xi_m\in\cotangent{m}$.
	Consequently, we obtain
	\begin{equation*}
		\begin{aligned}
			F_{mmm}^{***} (\xi_m)
			&
			=
			\sup_{p \in \cM} \paren[auto]\{\}{\dual{\xi_m}{\logarithm{m} p} - f_m^{**}(\logarithm{m} p)}
			\\
			&
			=
			\sup_{X \in \tangent{m}} \left\{ \dual{\xi_m}{X} - f_m^{**}(X)\right\}
			\\
			&
			=
			f_m^{***}(\xi_m)
			.
		\end{aligned}
	\end{equation*}
	According to \cite[Prop.~13.14~(iii)]{BauschkeCombettes:2011:1}, we have $f_m^{***} = f_m^*$.
	Collecting all equalities confirms \eqref{eq:m_Triconjugate}.
\end{proof}

The following is the analogue of \Cref{item:Classical_properties_Fenchel-Young} in \cref{lemma:Classical_properties_Fenchel_conjugate}.
\begin{proposition}[Fenchel--Young inequality]
	\label{prop:Manifold_Fenchel_Young_inequality}%
	Suppose that $\cC \subset \cM$ is strongly convex.
	Let $F\colon\cC \to\eR$ be proper and $m \in \cC$.
	Then
	\begin{equation}
		\label{eq:Manifold_Fenchel_Young_inequality}
		F(p) + F_m^*(\xi_m)
		\geq
		\dual{\xi_m}{\logarithm{m} p}
	\end{equation}
	holds for all~$p\in \cC$ and~$\xi_m\in\cotangent{m}$.
\end{proposition}
\begin{proof}
	If $F(p) = \infty$ the inequality trivially holds, since $F$ is proper and hence $F^*$ is nowhere $-\infty$.
	It remains to consider $F(p) <\infty$.
	Suppose that $\xi_m\in \cotangent{m}$, $p \in \cC$ and set $X \coloneqq \logarithm{m} p$.
	From~\cref{def:Fenchel_conjugate_on_M} we obtain
	\begin{equation*}
		F_m^*(\xi_m)
		\geq
		\dual{\xi_m}{\logarithm{m} p}
		-
		F \paren[big](){\exponential{m} \logarithm{m} p},
	\end{equation*}
	which is equivalent to~\eqref{eq:Manifold_Fenchel_Young_inequality}.
\end{proof}

We continue by introducing the manifold counterpart of the Fenchel--Moreau Theorem, compare \cref{thm:Classical_Fenchel_Moreaou_theorem}.
Given a set $\cC \subset \cM$, $m \in \cC$ and a function $F \colon \cC \to \eR$, we define $f_m\colon \tangent{m} \to \eR$ by
\begin{equation}\label{eqref:fcomp1}
	f_m(X) =
	\begin{cases}
		F(\exponential{m} X), & X \in \LocalExponentialPreimage{m}{\cC},
		\\
		\quad +\infty, & X \notin \LocalExponentialPreimage{m}{\cC}.
	\end{cases}
\end{equation}
Throughout this section, the convexity of the function $f_m \colon \tangent{m} \to \eR$ is the usual convexity on the vector space $\tangent{m}$, \ie, for all $X,Y\in \tangent{m}$ and $\lambda\in [0,1]$ it holds
\begin{equation} \label{eq:con.com.fun}
	f_m\paren[big](){(1-\lambda)X + \lambda Y} \leq (1-\lambda)f_m(X) + \lambda f_m(Y).
\end{equation}
We present two examples of functions~$F \colon \cM \to \R$ defined on Hadamard manifolds such that $f_m$ is convex.
In the first example, $F$ depends on an arbitrary fixed point $m'\in \cM$. 
In this case, we can guarantee that $f_m$ is convex only when $m = m'$. 
In the second example, $F$ is defined on a particular Hadamard manifold and $f_m$ is convex for any base point $m \in \cM$.
It is worth emphasizing that the functions in the following examples are geodesically convex as well but in general, the convexity of $F$ and $f_m$ are unrelated and all four cases can occur.

\begin{example}
	\label{example:convexity_on_the_tangent_space}
	Let $\cM$ be any Hadamard manifold and $m'\in \cM$ arbitrary.
	Consider the function $f_{m'}$ defined in \eqref{eqref:fcomp1} with
	$F\colon{\cM} \to \R$ given by $F(p) = \dist{m'}{p}$ for all $p\in\cM$.
	Note that
	\begin{equation*}
		f_{m'}(X)
		=
		F(\exponential{m'} X)
		=
		\dist[big]{m'}{\exponential{m'} X}
		=
		\riemanniannorm{X}[m'] 
		\quad 
		\text{for all } X \in \tangent{m'}.
	\end{equation*}
	Hence, it is easy to see that $f_{m'}$ satisfies \eqref{eq:con.com.fun} and, consequently, it is convex on $\tangent{m'}$.
\end{example}

Our second example is slightly more involved.
A problem involving the special case $a = 0$ and $b = 1$ appears in the dextrous hand grasping problem in \cite[Sect.~3.4]{DirrHelmkeLageman:2007:1}.
\begin{example} \label{ex:psdms}
	Denote by $\cP(n)$ the set of symmetric matrices of size $n\times n$ for some $n \in \N$, and by $\cM = \cP_+(n)$ the cone of symmetric positive definite matrices.
	The latter is endowed with the affine invariant Riemannian metric, given by
	\begin{equation}\label{eq:metric}
		\riemannian{X}{Y}[p] 
		\coloneqq 
		\trace (Xp^{-1}Yp^{-1})
		\quad
		\text{for }
		p\in \cM
		\text{ and }
		X,Y\in \tangent{p}[\cM]
		.
	\end{equation}
	The tangent space $\tangent{p}[\cM]$ can be identified with $\cP(n)$.
	$\cM$ is a Hadamard manifold,  see  for example \cite[Thm.~1.2, p.~325]{Lang:1999:1}.
	The exponential map $\exponential{p} \colon \tangent{p}[\cM] \to \cM$ is given by
	\begin{equation} \label{eq:GoedSPD}
		\exponential{p}{X}
		=
		p^{1/2} \e^{(p^{-1/2}Xp^{-1/2})} p^{1/2}
		\quad
		\text{for } 
		(p,X)\in \tangentBundle[\cM]
		.
	\end{equation}
	Consider the function $F \colon \cM \to \R$, defined by
	\begin{equation}  \label{eq:example.referee}
		F(p)
		= 
		a \ln^2(\det p ) - b\ln(\det p),
	\end{equation}
	where $a \ge 0$ and $b \in \R$ are constants. 
	Using \eqref{eq:GoedSPD} and properties of $\det \colon \cP(n) \to \R$, we have
	\begin{align*}
		\det(\exponential{m}{X})
		&
		= 
		\det \e^{(m^{-1/2}Xm^{-1/2})} \det m
		\\
		&
		= 
		\e^{\trace(m^{-1/2}Xm^{-1/2})} \det m 
		= 
		\e^{\trace(m^{-1}X)} \det m
		,
	\end{align*}
	for any $m\in \cM$. 
	Hence, considering $f_m(X) = F(\exponential{m} X)$, we obtain
	\begin{multline*}
		f_m(X)
		=  
		a \trace^2(m^{-1}X ) + 2 a \trace(m^{-1}X) \ln( \det m) +  a \ln^2( \det m) 
		\\
		- b \trace (m^{-1}X) - b \ln(\det m)
		,
	\end{multline*}
	for any $m\in \cM$. 
	The Euclidean gradient and Hessian of $f_m$ are given by
	\begin{align*}
		f_m'(X) 
		&
		= 
		2a \trace(m^{-1}X) m^{-1} + 2a\ln(\det m) m^{-1} - bm^{-1}
		,  
		\\
		f_m''(X)(Y,\cdot)
		&
		= 
		2a \trace(m^{-1}Y ) m^{-1}
		,
	\end{align*}
	respectively, for all $X,Y \in \cP(n)$. 
	Hence $f_m''(X)(Y,Y) = 2a \trace^2(m^{-1}Y) \ge 0$ holds.
	Thus, the function $f_m$ is convex for any $m \in \cM$.
	From \cite[Ex.~4.4]{FerreiraLouzeiroPrudente:2019:1} we can conclude that \eqref{eq:example.referee} is  also geodesically convex.
\end{example}

Since $\paren[big](){\tangent{m},\riemannian{\cdot}{\cdot}[m]}$ is a Hilbert space, the function $f_m$ defined in \eqref{eqref:fcomp1} establishes a relationship between the results of this section and the results of \cref{subsec:Convex_Analysis}.
We will exploit this relationship in the demonstration of the following results.
\begin{lemma}\label{lem:pro.Fm}
	Suppose that $\cC \subset \cM$ is strongly convex and $m \in \cC$.
	Suppose that $F \colon \cC \to \eR$.
	Then the following statements hold:
	\begin{enumerate}
		\item
			\label{lem:pro.Fm:proper}
			$F$ is proper if and only if $f_m$ is proper.
		\item
			\label{lem:pro.Fm:eq}
			$F_m^*(\xi) = f_m^*(\xi)$ for all $\xi \in \cotangent{m}$.
		\item
			\label{lem:pro.Fm:convex_lsc}
			The function $F_m^*$ is convex and lsc on $\cotangent{m}$.
		\item
			\label{lem:pro.Fm:biconj}
			$F_{m m}^{**}(p) = f_m^{**}(\logarithm{m} p)$ for all $p \in \cC$.
	\end{enumerate}
\end{lemma}
\begin{proof}
	Since $\cC\subset\cM$ is strongly convex, \cref{lem:pro.Fm:proper} follows directly from \eqref{eqref:fcomp1} and the fact that the map $\exponential{m} \colon \LocalExponentialPreimage{m}{\cC} \to \cC$ is bijective.
	As for \cref{lem:pro.Fm:eq}, \cref{def:Fenchel_conjugate_on_M} and the definition of $f_m$ in \eqref{eqref:fcomp1} imply
	\begin{align*}
		F_m^*(\xi)
		&
		=
		\sup_{X \in \LocalExponentialPreimage{m}{\cC}} \paren[auto]\{\}{\dual{\xi}{X} - F(\exponential{m} X)}
		=
		- \inf_{X \in \LocalExponentialPreimage{m}{\cC}} \paren[auto]\{\} {F(\exponential{m} X) - \dual{\xi}{X}}
		\\
		&
		=
		-\inf_{X \in \tangent{m}} \paren[auto]\{\}{f_m(X) - \dual{\xi}{X}}
		=
		\sup_{X \in \tangent{m} } \paren[auto]\{\}{\dual{\xi}{X} - f_m(X)}
		=
		f_m^*(\xi)
	\end{align*}
	for all $\xi \in \cotangent{m}$.
	\cref{lem:pro.Fm:convex_lsc} follows immediately from \cite[Prop.~13.11]{BauschkeCombettes:2011:1} and \cref{lem:pro.Fm:eq}.
	For \cref{lem:pro.Fm:biconj}, take $p \in \cC$ arbitrary.
	Using \cref{definition:biconjugate} and \cref{lem:pro.Fm:eq} we have
	\begin{align*}
		F_{m m}^{**}(p)
		&
		=
		\sup_{\xi\in \cotangent{m}} \paren[auto]\{\}{\dual{\xi}{\log_mp} -F_m^*(\xi)}
		\\
		&
		=
		\sup_{\xi\in \cotangent{m}} \paren[auto]\{\}{\dual{\xi} {\log_mp} -f_m^*(\xi)}
		=
		f_m^{**}(\xi),
	\end{align*}
	which concludes the proof.
\end{proof}
In the following theorem we obtain a version of the Fenchel--Moreau \cref{thm:Classical_Fenchel_Moreaou_theorem} for functions defined on Riemannian manifolds.
To this end, it is worth noting that if $\cC$ is strongly convex then
\begin{equation}\label{eq:pre.mai.theo}
	F(p)
	=
	f_m(\log_mp)
	\quad \text{ for all } p \in \cC.
\end{equation}
Equality \eqref{eq:pre.mai.theo} is an immediate consequence of \eqref{eqref:fcomp1}, and will be used in the proof of the following two theorems.

\begin{theorem}\label{theo:fen.mor.man}
	Suppose that $\cC \subset \cM$ is strongly convex and $m \in \cC$.
	Let $F \colon \cC \to \eR$ be proper.
	If $f_m$ is lsc and convex on $\tangent{m}$, then $F = F_{m m}^{**}$.
	In this case $F^*_m$ is proper as well.
\end{theorem}
\begin{proof}
	First note that due to \cref{lem:pro.Fm}~\cref{lem:pro.Fm:proper}, the function $f_m$ is also proper.
	Taking into account \cref{thm:Classical_Fenchel_Moreaou_theorem}, it follows that $f_m = f_m^{**}$.
	Thus, considering \eqref{eq:pre.mai.theo}, we have $F(p) = f_m^{**}(\log_mp) $ for all $p \in \cC$.
	Using \cref{lem:pro.Fm}~\cref{lem:pro.Fm:biconj} we can conclude that $F = F_{m m}^{**}$.
	Furthermore by \cref{lem:pro.Fm}~\cref{lem:pro.Fm:proper}, $f_m$ is proper.
	Hence by \cref{thm:Classical_Fenchel_Moreaou_theorem}, we obtain that $f_m^*$ is proper and by \cref{lem:pro.Fm}~\cref{lem:pro.Fm:eq}, $F_m^*$ is proper as well.
\end{proof}

\begin{theorem}\label{theo:fen.mor.man2}
	Suppose that $\cM$ is a Hadamard manifold and $m \in \cM$.
	Suppose that $F \colon \cM \to \eR$ is a proper function.
	Then $f_m$ is lsc and convex on $\tangent{m}$ if and only if $F=F_{m m}^{**}$.

	In this case $F_m^*$ is proper as well.
\end{theorem}
\begin{proof}
	Observe that due to \cref{lem:pro.Fm}~\cref{lem:pro.Fm:proper}, the function $f_m$ is proper.
	Taking into account \cref{thm:Classical_Fenchel_Moreaou_theorem},
	it follows that $f_m$ is lsc and convex on $\tangent{m}$ if and only if $f_m = f_m^{**}$.
	Considering \eqref{eq:pre.mai.theo} and \cref{lem:pro.Fm}~\cref{lem:pro.Fm:biconj}, both with $\cC={\cM}$, we can say that $f_m = f_m^{**}$ is equivalent to $F = F_{m m}^{**}$.
	Properness of $F_m^*$ follows by the same arguments as in \cref{theo:fen.mor.man}.
	This completes the proof.
\end{proof}

We now address the manifold counterpart of \cref{thm:Classical_characterization_subdifferential}, whose proof is a minor extension compared to the proof for \cref{thm:Classical_characterization_subdifferential} and therefore omitted.
\begin{theorem}
	\label{thm:Manifold_Characterization_of_the_subdifferential}%
	Suppose that $\cC \subset \cM$ is strongly convex and $m, p \in \cC$.
	Let $F \colon \cC \to \eR$ be a proper function.
	Suppose that $f_m$ defined in \eqref{eqref:fcomp1} is convex on $\tangent{m}$.
	Then $\parallelTransport{p}{m}{\xi_p} \in \partial f_m(\logarithm{m} p)$ if and only if
	\begin{equation}
		\label{eq:RHS_equivalence_being_in_the_subdifferential}
		f_m(\logarithm{m} p) + f_m^*(\parallelTransport{p}{m}{\xi_p})
		=
		\dual[auto]{\parallelTransport{p}{m}{\xi_p}}{\logarithm{m} p}.
	\end{equation}
\end{theorem}

Given $F \colon \cC \to \eR$ and $m\in\cC$, we can state the subdifferential from \cref{def:Subdifferential} for the Fenchel $m$-conjugate function $F_m^* \colon \cotangent{m} \to \eR$.
Note that $F_m^*$ is convex by \cref{lem:pro.Fm}~\cref{lem:pro.Fm:convex_lsc} and defined on the cotangent space $\cotangent{m}$, so the following
equation is a classical subdifferential written in terms of tangent vectors, since the dual space of $\cotangent{m}$ can be canonically identified with $\tangent{m}$. The subdifferential definition reads as follows:
\makeatletter
\IfSubStr{\@currentclass}{svjour3.cls}{%
	\begin{multline*}
		\partial F_m^*(\xi_m)
		\coloneqq
		\\
		\setDef[auto]{X \in \tangent{m}}{F_m^*(\eta_m) \geq F_m^*(\xi_m) + \dual{X}{\eta_m-\xi_m} \text{ for all } \eta_m\in \cotangent{m}}
		.
	\end{multline*}
	}{%
	\begin{equation*}
		\partial F_m^*(\xi_m)
		\coloneqq
		\setDef[auto]{X \in \tangent{m}}{F_m^*(\eta_m) \geq F_m^*(\xi_m) + \dual{X}{\eta_m-\xi_m} \text{ for all } \eta_m\in \cotangent{m}}
		.
	\end{equation*}
}
\makeatother
Before providing the manifold counterpart of \cref{cor:Classical_symmetry_subdifferential}, let us show how~\cref{thm:Manifold_Characterization_of_the_subdifferential} reads for~$F_m^*$.

\begin{corollary}
	\label{cor:Manifold_characterization_conjugate_function}%
	Suppose that $\cC \subset \cM$ is strongly convex and $m, p \in \cC$.
	Let $F\colon \cC\to\eR$ be a proper function and let $f_m$ be the function defined in \eqref{eqref:fcomp1}.
	Then
	\begin{equation}
		\label{eq:Manifold_characterization_conjugate_function}
		\logarithm{m} p \in \partial F_m^* (\xi_m)
		\quad \Leftrightarrow \quad
		F_m^*(\xi_m) + f_m(\logarithm{m} p)
		=
		\dual{\xi_m}{\logarithm{m} p}
	\end{equation}
	holds for all $\xi_m \in \cotangent{m}$.
\end{corollary}
\begin{proof}
	The proof follows directly from the fact that $F_m^*$ is defined on the vector space $\cotangent{m}$ and that $F_m^*$ is convex due to \cref{lem:pro.Fm}~\cref{lem:pro.Fm:convex_lsc}.
\end{proof}
To conclude this section, we state the following result, which generalizes \cref{cor:Classical_symmetry_subdifferential} and shows the symmetric relation between the conjugate function and the subdifferential when the function involved is proper, convex and lsc.
\begin{corollary}
	\label{cor:Manifold_symmetric_property_of_the_subdifferential}%
	Let $F\colon \cC\to\eR$ be a proper function and $m,p\in\cC$.
	If the function $f_m$ defined in \eqref{eqref:fcomp1} is convex and lsc on $\tangent{m}$, then
	\begin{equation}
		\label{eq:Manifold_symmetric_property_of_the_subdifferential}
		\parallelTransport{p}{m}{\xi_p} \in \partial f_m(\logarithm{m} p)
		\quad \Leftrightarrow \quad
		\logarithm{m} p \in \partial F_m^* (\parallelTransport{p}{m}{\xi_p})
		.
	\end{equation}
\end{corollary}
\begin{proof}
	The proof is a straightforward combination of~\cref{thm:Manifold_Characterization_of_the_subdifferential,theo:fen.mor.man} and taking as a particular cotangent vector $\xi_m = \parallelTransport{p}{m}{\xi_p}$ in~\cref{cor:Manifold_characterization_conjugate_function}.
\end{proof}

\section{Optimization on Manifolds}
\label{sec:Primal-Dual_on_manifolds}

In this section we derive a primal-dual optimization algorithm to solve mini\-mi\-za\-tion problems on Riemannian manifolds of the form
\begin{equation}
	\label{eq:General_problem_for_PD_Chambolle-Pock}
	\text{Minimize} \quad \Fidelity(p) + \Prior(\Lambda(p)), \quad p \in \cC.
\end{equation}
Here $\cC\subset \cM$ and $\cD \subset \cN$ are strongly convex sets, $\Fidelity\colon \cC \to \eR$ and~$\Prior\colon \cD \to \eR$ are proper functions, and $\Lambda \colon \cM \to \cN$ is a general differentiable map such that $\Lambda(\cC) \subset \cD$.
Furthermore, we assume that $\Fidelity\colon \cC \to \eR$ is geodesically convex and that
\begin{equation}\label{eqref:gcomp1}
	g_n(X) =
	\begin{cases}
		G(\exponential{n}{X}), & X \in \LocalExponentialPreimage{n}{\cD},
		\\
		\quad +\infty, & X \notin \LocalExponentialPreimage{n}{\cD},
	\end{cases}
\end{equation}
is proper, convex and lsc on $\tangent{n}[\cN]$ for some $n \in \cD$.
One model that fits these requirements is the dextrous hand grasping problem from \cite[Sect.~3.4]{DirrHelmkeLageman:2007:1}.
There $\cM = \cN = \cP_+(n)$ is the Hadamard manifold of symmetric positive matrices, $F(p) = \trace(wp)$ holds with some $w \in \cM$, and $G(p) = -\log\det(p)$, cf.~\cref{ex:psdms}.
Another model verifying the assumptions will be presented in \cref{sec:ROF_Models}.

Our algorithm requires a choice of a pair of base points $m \in \cC$ and $n \in \cD$.
The role of $m$ is to serve as a possible linearization point for $\Lambda$, while $n$ is the base point of the Fenchel conjugate for $G$.
More generally, the points can be allowed to change during the iterations.
We emphasize this possibility by writing $m^{(k)}$ and $n^{(k)}$ when appropriate.

Under the standing assumptions, the following saddle-point formulation is equivalent to~\eqref{eq:General_problem_for_PD_Chambolle-Pock}:
\begin{equation}
	\label{eq:Manifold_saddle-point_representation_for_PD_Chambolle-Pock}
	\text{Minimize} \quad
	\sup_{\xi_n \in \cotangent{n}[\cN]} \dual[auto]{\logarithm{n} \Lambda(p)}{\xi_n} + \Fidelity(p) - \Prior_n^*(\xi_n),
	\quad
	p \in \cC
	.
\end{equation}
The proof of equivalence uses \cref{theo:fen.mor.man} applied to $G$ and the details are left to the reader.

From now on, we will consider problem \eqref{eq:Manifold_saddle-point_representation_for_PD_Chambolle-Pock}, whose solution by primal-dual optimization algorithms is challenging due to the lack of a vector space structure, which implies in particular the absence of a concept of linearity of $\Lambda$.
This is also the reason why we cannot derive a dual problem associated with \eqref{eq:General_problem_for_PD_Chambolle-Pock} following the same reasoning as in vector spaces.
Therefore we concentrate on the saddle-point problem \eqref{eq:Manifold_saddle-point_representation_for_PD_Chambolle-Pock}.
Following along the lines of \cite[Sect.~2]{Valkonen:2014:1}, where a system of optimality conditions for the Hilbert space counterpart of the saddle-point problem \eqref{eq:Manifold_saddle-point_representation_for_PD_Chambolle-Pock} is stated, we conjecture that if $\paren[big](){\widehat p, \widehat \xi_n} \in \cC \times \cotangent{n}[\cN]$ solves \eqref{eq:Manifold_saddle-point_representation_for_PD_Chambolle-Pock}, then it satisfies the system
\begin{equation}
	\label{eq:First_order_optimality_condition_for_general_saddle_point_problem_with_K_nonlinear_at_solution}
	\begin{aligned}
		-D \Lambda(\widehat{p})^*[\parallelTransport{n}{\Lambda(\widehat{p})}{\widehat \xi_n}]
		&
		\in \partial_\cM F(\widehat{p}),
		\\
		\logarithm{n} \Lambda(\widehat{p})
		&
		\in \partial G_n^*(\widehat \xi_n)
		.
	\end{aligned}
\end{equation}

Motivated by \cite[Sect.~2.2]{Valkonen:2014:1} we propose
to replace $\widehat{p}$ by $m$, the point where we linearize the operator~$\Lambda$, which suggests to consider the system
\begin{equation}
	\label{eq:First_order_optimality_condition_for_general_saddle_point_problem_with_K_nonlinear}
	\begin{aligned}
		\parallelTransport[big]{m}{p}(-D \Lambda(m)^*[\parallelTransport{n}{\Lambda(m)}{\xi_n}])
		&
		\in \partial_\cM F({p}),
		\\
		\logarithm{n} \Lambda({p})
		&
		\in \partial G_n^*(\xi_n)
		,
	\end{aligned}
\end{equation}
for the unknowns $(p,\xi_n)$.

\begin{remark}\label{rem:Particularization_eRCP_Euclidean_setting}
	In the specific case that $\cX=\cM$ and~$\cY=\cN$ are Hilbert spaces, $F\colon \cX\to \R$ is continuously differentiable, $\Lambda \colon \cX \to \cY$ is a linear operator, $m=n=0$, and either $D\Lambda(m)^*$ has empty null space or $\dom G = \cY$, we observe (similar to~\cite{Valkonen:2014:1}) that the conditions~\eqref{eq:First_order_optimality_condition_for_general_saddle_point_problem_with_K_nonlinear} simplify to
	\begin{equation}
		\label{eq:First_order_optimality_condition_for_general_saddle_point_problem_with_K_linear}
		\begin{aligned}
			-\Lambda^* \xi
			&
			\in \partial F(p)
			,
			\\
			\Lambda p
			&
			\in \partial G^*(\xi)
			,
		\end{aligned}
	\end{equation}
	where~$ p \in \cX$ and $ \xi \in \cotangent{n}[\cN] = \cY^*$.
\end{remark}

\subsection{Exact Riemannian Chambolle--Pock}
\label{subsec:Exact_Riemannian_PD_CP}

\begin{algorithm}[tbp]
	\caption{Exact (primal relaxed) Riemannian Chambolle--Pock for~\eqref{eq:Manifold_saddle-point_representation_for_PD_Chambolle-Pock}}
	\label{alg:PrimalOverrelax-eRCP}
	\begin{algorithmic}[1]
		\Require
		$m\in\cC$,
		$n\in\cD$,
		$p^{(0)} \in \cC$,
		$\xi_n^{(0)} \in \cotangent{n}[\cN]$,
		and parameters $\primalstep_0$, $\dualstep_0$, $\theta_0$, $\gamma$
		\State $k \gets 0$, \quad $\bar p^{(0)} \gets p^{(0)}$
		\While{not converged}
		\State
		\label{alg:exact:CP:DualStep}
		$\xi^{(k+1)}_n \gets \prox[Big]{\dualstep_k \Prior_n^*}{ \xi_n^{(k)} + \dualstep_k \paren[auto](){\logarithm[auto]{n} \Lambda \paren[big](){\bar p^{(k)}}}^\flat}$,
		\State
		\label{alg:exact:CP:PrimalStep}
		$p^{(k+1)} \gets \prox[Big]{\primalstep_k \Fidelity}{\exponential[auto]{p^{(k)}}{\parallelTransport{m}{p^{(k)}}{\paren[big](){-\primalstep_k D\Lambda(m)^*\paren[big][]{\parallelTransport{n}{\Lambda(m)}{\xi_n^{(k+1)}}}}}^\sharp}}$,
		\State
		$\theta_k = (1+2\gamma\primalstep_k)^{-\frac{1}{2}}$,
		\quad
		$\primalstep_{k+1} \gets \primalstep_k\theta_k$,
		\quad
		$\dualstep_{k+1} \gets \dualstep_k/\theta_k$
		\State
		\label{alg:exact:CP:UpdateStep_primal}%
		$\bar p^{(k+1)} \gets \exponential[big]{p^{(k+1)}}{-\theta_k \, \logarithm{p^{(k+1)}} p^{(k)} }$
		\State $k \gets k+1$
		\EndWhile
		\Ensure $p^{(k)}$
	\end{algorithmic}
\end{algorithm}

In this subsection we develop the \emph{exact} Riemannian Chambolle--Pock algorithm summarized in~\cref{alg:PrimalOverrelax-eRCP}. The name \eqq{exact}, introduced by~\cite{Valkonen:2014:1}, refers to the fact that the operator~$\Lambda$ in the dual step is used in its exact form and only the primal step employs a linearization in order to obtain the adjoint~$D\Lambda(m)^*$.
Indeed, our \cref{alg:PrimalOverrelax-eRCP} can be interpreted as generalization of \cite[Alg.~2.1]{Valkonen:2014:1}.

Let us motivate the formulation of \cref{alg:PrimalOverrelax-eRCP}.
We start from the second inclusion in \eqref{eq:First_order_optimality_condition_for_general_saddle_point_problem_with_K_nonlinear} and obtain, for any $\dualstep > 0$, the equivalent condition
\begin{equation}
	\label{eq:Formula_for_p_in_CP}
	\xi_n + \dualstep \paren[auto](){\logarithm{n} \Lambda({p})}^\flat
	\in
	\xi_n + \paren[auto](){\dualstep \partial G_n^*( \xi_n)}^\flat
	=
	\paren[big](){\id +(\dualstep \partial G^*_n)^\flat}( \xi_n).
\end{equation}
Similarly we obtain that the first inclusion in~\eqref{eq:First_order_optimality_condition_for_general_saddle_point_problem_with_K_nonlinear} is equivalent to
\begin{equation}
	\label{eq:Motivation_primal_prox}
	- \frac{1}{\primalstep} \paren[auto](){\primalstep \parallelTransport{m}{p}{D \Lambda(m)^*[\parallelTransport{n}{\Lambda(m)}{ \xi_n}]}} \in \partial_\cM F(p)
\end{equation}
for any $\primalstep>0$.
\cref{lem:proxSubdiff} now suggests the following alternating algorithmic scheme:
\begin{subequations}
	\begin{align}
		\xi_n^{(k+1)}
		&
		=
		\prox[big]{\dualstep \Prior^*_n}{\widetilde \xi_n^{(k)}}
		,
		\nonumber
		\\
		p^{(k+1)}
		&
		=
		\prox[big]{\primalstep F}{\widetilde p^{(k)}}
		,
		\nonumber
		\intertext{where}
		\widetilde\xi_n^{(k)}
		&
		\coloneqq \xi_n^{(k)} + \dualstep \paren[Big](){\logarithm[big]{n} \Lambda(\bar p^{(k)})}^\flat
		,
		\\
		\widetilde p^{(k)}
		&
		\coloneqq
		\exponential[auto]{p^{(k)}}{\parallelTransport{m}{p^{(k)}}{-\paren[big](){\primalstep D\Lambda(m)^*\paren[big][]{\parallelTransport{n}{\Lambda(m)}{\xi_n^{(k+1)}}}}^\sharp}},
		\label{eq:widetilde_xi}
		\\
		\bar p^{(k+1)}
		&
		=
		\exponential[big]{p^{(k+1)}}{-\theta \, \logarithm{p^{(k+1)}} p^{(k)}}
		.
		\label{eq:widetilde_p}
	\end{align}
\end{subequations}
Through~$\theta$ we perform an over-relaxation of the primal variable.
This basic form of the algorithm can be combined with an acceleration by step size selection as described in~\cite[Sec.~5]{ChambollePock:2011:1}.
This yields \cref{alg:PrimalOverrelax-eRCP}.

\subsection{Linearized Riemannian Chambolle--Pock}
\label{subsec:Linearized Riemannian Primal-Dual Chambolle-Pock}

The main obstacle in deriving a complete duality theory for problem \eqref{eq:Manifold_saddle-point_representation_for_PD_Chambolle-Pock} is the lack of a concept of linearity of operators $\Lambda$ between manifolds.
In the previous section, we chose to linearize $\Lambda$ in the primal update step only, in order to have an adjoint.
By contrast, we now replace $\Lambda$ by its first order approximation
\begin{equation}
	\label{eq:First_order_approximation_of_Lambda}
	\Lambda (p) \approx \exponential[auto]{\Lambda(m)}{D\Lambda(m)[\logarithm{m} p]}
\end{equation}
everywhere throughout this section.
Here $D\Lambda(m)\colon \tangent{m}\to \tangent{\Lambda(m)}[\cN]$ denotes the derivative (push-forward) of~$\Lambda$ at~$m$.
Since~$D\Lambda\colon \tangentBundle[\cM] \to \tangentBundle[\cN]$ is a linear operator between tangent bundles, we can utilize the adjoint operator $D\Lambda(m)^* \colon \cotangent{\Lambda(m)}[\cN] \to \cotangent{m}$.
We further point out that we can work algorithmically with cotangent vectors $\xi_n \in \cotangent{n}[\cN]$ with a fixed base point~$n$ since, at least locally, we can obtain a cotangent vector~$\xi_{\Lambda(m)}\in \cotangent{\Lambda(m)}[\cN]$ from it by parallel transport using $\xi_{\Lambda(m)} = \parallelTransport{n}{\Lambda(m)}{\xi_n}$.
The duality pairing reads as follows:
\begin{equation}
	\label{eq:Duality_pairing_for_DLambda}
	\dual[auto]{D\Lambda(m)[\logarithm{m} p]}{\parallelTransport{n}{\Lambda(m)}{\xi_n}}
	=
	\dual[auto]{\logarithm{m} p}{D\Lambda(m)^*[\parallelTransport{n}{\Lambda(m)}{\xi_n}]}
\end{equation}
for every~$p\in \cC $ and $\xi_n \in \cotangent{n}[\cN]$.

We substitute the approximation~\eqref{eq:First_order_approximation_of_Lambda} into~\eqref{eq:General_problem_for_PD_Chambolle-Pock}, which
yields the \emph{linearized} primal problem
\begin{equation}\label{eq:Linearized_primal_problem}
	\text{Minimize}
	\quad
	\Fidelity(p) + \Prior\paren[big](){\exponential[auto]{\Lambda(m)}{D\Lambda(m)[\logarithm{m} p]}},
	\quad
	p \in \cC
	.
\end{equation}
For simplicity, we assume $\Lambda(m)=n$ for the remainder of this subection.
Hence, the analogue of the saddle-point problem \eqref{eq:Manifold_saddle-point_representation_for_PD_Chambolle-Pock} reads as follows:
\begin{equation}
	\label{eq:Linearized_manifold_saddle-point_reprsentation_for_PD_Chambolle-Pock}
	\text{Minimize} \quad \sup_{\xi_n \in \cotangent{n}[\cN]} \dual[auto]{D\Lambda(m)[\logarithm{m} p]}{\xi_n} + \Fidelity(p) - \Prior_n^*(\xi_n)
	,
	\quad
	p
	\in \cC
	.
\end{equation}
We refer to it as the \emph{linearized} saddle-point problem.
Similar as for \eqref{eq:General_problem_for_PD_Chambolle-Pock} and \eqref{eq:Manifold_saddle-point_representation_for_PD_Chambolle-Pock}, problems \eqref{eq:Linearized_primal_problem} and \eqref{eq:Linearized_manifold_saddle-point_reprsentation_for_PD_Chambolle-Pock} are equivalent by \cref{theo:fen.mor.man}.
In addition, in contrast to~\eqref{eq:General_problem_for_PD_Chambolle-Pock}, we are now able to also derive a Fenchel dual problem associated with~\eqref{eq:Linearized_primal_problem}.
\begin{theorem}\label{thm:weak-duality}
	The dual problem of~\eqref{eq:Linearized_primal_problem} is given by
	\begin{equation}\label{eq:Linearized_dual_problem}
		\text{Maximize}
		\quad
		- \Fidelity_m^* \paren[big](){-D\Lambda(m)^*[\xi_n]} - \Prior_n^*(\xi_n),
		\quad
		\xi_n\in\cotangent{n}[\cN]
		.
	\end{equation}
	Weak duality holds, \ie,
	\makeatletter
	\IfSubStr{\@currentclass}{svjour3.cls}{%
		\begin{multline}\label{eq:Weak_duality_linearized_pair}
			\inf_{p\in\cC} \paren[auto]\{\}{\Fidelity(p) + \Prior\paren[big](){\exponential[auto]{\Lambda(m)}{D\Lambda(m)[\logarithm{m} p]}}}
			\\
			\geq
			\sup_{\xi_n\in\cotangent{n}[\cN]} \paren[auto]\{\}{-\Fidelity_m^*\paren[big](){-D\Lambda(m)^*[\xi_n]} - \Prior_n^*(\xi_n)}
			.
		\end{multline}
		}{%
		\begin{equation}\label{eq:Weak_duality_linearized_pair}
			\inf_{p\in\cC} \paren[auto]\{\}{\Fidelity(p) + \Prior\paren[big](){\exponential[auto]{\Lambda(m)}{D\Lambda(m)[\logarithm{m} p]}}}
			\geq
			\sup_{\xi_n\in\cotangent{n}[\cN]} \paren[auto]\{\}{-\Fidelity_m^*\paren[big](){-D\Lambda(m)^*[\xi_n]} - \Prior_n^*(\xi_n)}
			.
		\end{equation}
	}
	\makeatother
\end{theorem}
\begin{proof}
	The proof of \eqref{eq:Linearized_dual_problem} and \eqref{eq:Weak_duality_linearized_pair} follows from the application of~\cite[eq.(2.80)]{Zalinescu:2002:1} and \cref{def:Fenchel_conjugate_on_M} in \eqref{eq:Linearized_manifold_saddle-point_reprsentation_for_PD_Chambolle-Pock}.
\end{proof}

Notice that the analogue of \eqref{eq:First_order_optimality_condition_for_general_saddle_point_problem_with_K_nonlinear} is
\begin{equation} \label{eq:First_order_optimality_condition_for_linearized_primal_dual_with_K_nonlinear}
	\begin{aligned}
		\parallelTransport[big]{m}{p}(-D \Lambda(m)^*[ \xi_n])
		&
		\in \partial_\cM F({p})
		,
		\\
		D \Lambda(m)[\logarithm{m} p]
		&
		\in \partial G_n^*( \xi_n)
		.
	\end{aligned}
\end{equation}
In the situation described in \cref{rem:Particularization_eRCP_Euclidean_setting}, \eqref{eq:First_order_optimality_condition_for_linearized_primal_dual_with_K_nonlinear} agrees with \eqref{eq:First_order_optimality_condition_for_general_saddle_point_problem_with_K_linear}.
Motivated by the statement of the linearized primal-dual pair \eqref{eq:Linearized_primal_problem}, \eqref{eq:Linearized_dual_problem} and saddle-point system \eqref{eq:Linearized_manifold_saddle-point_reprsentation_for_PD_Chambolle-Pock}, a further development of duality theory and an investigation of the linearization error is left for future research.

Both the exact and the linearized variants of our Riemannian Chambolle--Pock algorithm (RCPA) can be stated in two variants, which  over-relax either the primal variable as in~\cref{alg:PrimalOverrelax-eRCP}, or the dual variable as in \cref{alg:DualOverrelax-lRCP}.
In total this yields four possibilities --- exact vs.\ linearized, and primal vs.\ dual over-relaxation.
This generalizes the analogous cases discussed in \cite{Valkonen:2014:1} for the Hilbert space setting.
In each of the four cases, it is possible to allow changes in the base points, and moreover, $n^{(k)}$ may be equal or different from $\Lambda(m^{(k)})$.
Letting $m^{(k)}$ depend on~$k$ changes the linearization point of the operator, while allowing $n^{(k)}$ to change introduces different $n^{(k)}$-Fenchel conjugates $\Prior_{n^{(k)}}^*$, and it also incurs a parallel transport on the dual variable.
These possibilities are reflected in the statement of \cref{alg:DualOverrelax-lRCP}.

Reasonable choices for the base points include, \eg, to set both~$m^{(k)}=m$ and $n^{(k)}=\Lambda(m)$, for $k\geq 0$ and some $m\in\cM$.
This choice eliminates the parallel transport in the dual update step as well as the innermost parallel transport of the primal update step.
Another choice is to fix just~$n$ and set~$m^{(k)} = p^{(k)}$, which eliminates the parallel transport in the primal update step.
It further eliminates both parallel transports of the dual variable in steps~\ref{alg:lin:CP:UpdateStep_dual} and~\ref{alg:linCP:PTdual} of \cref{alg:DualOverrelax-lRCP}.

\begin{algorithm}[tbp]
	\caption{Linearized (dual relaxed) Riemannian Chambolle--Pock for~\eqref{eq:Linearized_manifold_saddle-point_reprsentation_for_PD_Chambolle-Pock}}
	\label{alg:DualOverrelax-lRCP}
	\begin{algorithmic}[1]
		\Require
		$m^{(k)}\in\cC$,
		$n^{(k)}\in\cD$,
		$p^{(0)} \in \cC$,
		$\xi_n^{(0)} \in \cotangent{n^{(0)}[\cN]}$,
		and parameters $\primalstep_0$, $\dualstep_0$, $\theta_0$, $\gamma$
		\State
		$k \gets 0$, \quad $\bar p^{(0)} \gets p^{(0)}$
		\While{not converged}
		\State
		\label{alg:lin:CP:PrimalStep}
		$p^{(k+1)} \gets \prox[Big]{\primalstep_k\Fidelity}{\exponential[Big]{p^{(k)}}{\parallelTransport{m^{(k)}}{p^{(k)}}{\paren[big](){-\primalstep_k D\Lambda(m^{(k)})^* \paren[big][]{\parallelTransport{n^{(k)}}{\Lambda(m^{(k)})}{\bar \xi_{n^{(k)}}^{(k)}}}}^\sharp}}}$
		\State
		\label{alg:lin:CP:DualStep}
		$\xi_{n^{(k)}}^{(k+1)} \gets \prox[Big]{\dualstep_k \Prior_{\!n^{(k)}}^*}{\xi_{n^{(k)}}^{(k)} + \dualstep_k \paren[big](){\parallelTransport{\Lambda(m^{(k)})}{n^{(k)}}{D\Lambda(m^{(k)})[\logarithm{m^{(k)}} p^{(k+1)}]}}^\flat}$
		\State
		$\theta_k = (1+2\gamma\primalstep_k)^{-\frac{1}{2}}$,
		\quad
		$\primalstep_{k+1} \gets \primalstep_k\theta_k$,
		\quad
		$\dualstep_{k+1} \gets \dualstep_k/\theta_k$
		\State
		\label{alg:lin:CP:UpdateStep_dual}%
		$\bar \xi_{n^{(k+1)}}^{(k+1)}
		\gets
		\parallelTransport{n^{(k)}}{n^{(k+1)}}{\paren[auto](){\xi_{n^{(k)}}^{(k+1)} + \theta \, \paren[big](){\xi_{n^{(k)}}^{(k+1)}-\xi_{n^{(k)}}^{(k)}}}}$
		\State
		\label{alg:linCP:PTdual} $\xi_{n^{(k+1)}}^{(k+1)}
		\gets
		\parallelTransport{n^{(k)}}{n^{(k+1)}}{\xi_{n^{(k)}}^{(k+1)}}$
		\State $k \gets k+1$
		\EndWhile
		\Ensure
		$p^{(k)}$
	\end{algorithmic}
\end{algorithm}

\subsection{Relation to the Chambolle--Pock Algorithm in Hilbert Spaces}
\label{subsec:CP_in_Hilbert_space}

In this subsection we confirm that both \cref{alg:PrimalOverrelax-eRCP} and \cref{alg:DualOverrelax-lRCP} boil down to the classical Chambolle--Pock method in Hilbert spaces; see \cite[Alg.~1]{ChambollePock:2011:1}.
To this end, suppose in this subsection that $\cM = \cX$ and $\cN = \cY$ are finite-dimensional Hilbert spaces with inner products $\inner{\cdot}{\cdot}_\cX$ and $\inner{\cdot}{\cdot}_\cY$, respectively, and that $\Lambda\colon \cX \to \cY$ is a linear operator.
In Hilbert spaces, geodesics are straight lines in the usual sense.
Moreover, $\cX$ and $\cY$ can be identified with their tangent spaces at arbitrary points, the exponential map equals addition, and the logarithmic map equals subtraction.
In addition, all parallel transports are identity maps.

We are now showing that \cref{alg:PrimalOverrelax-eRCP} reduces to the classical Chambolle--Pock method when $n = 0 \in \cY$ is chosen.
The same then holds true for \cref{alg:DualOverrelax-lRCP} as well since $\Lambda$ is already linear.
Notice that the iterates $p^{(k)}$ belong to $\cX$ while the iterates $\xi^{(k)}$ belong to $\cY^*$.
We can drop the fixed base point $n = 0$ from their notation.
Also notice that $G_0^*$ agrees with the classical Fenchel conjugate and it will be denoted by $G^* \colon \cY \to \eR$.

We only need to consider steps~\ref{alg:exact:CP:DualStep},~\ref{alg:exact:CP:PrimalStep} and~\ref{alg:exact:CP:UpdateStep_primal} in \cref{alg:PrimalOverrelax-eRCP}.
The dual update step becomes
\begin{equation*}
	\xi^{(k+1)} \gets \prox[Big]{\dualstep_k \Prior^*}{ \xi^{(k)} + \dualstep_k \paren[big](){\Lambda \bar p^{(k)}}^\flat}
	.
\end{equation*}
Here $\flat \colon \cY \to \cY^*$ denotes the Riesz isomorphism for the space $\cY$.
Next we address the primal update step, which reads
\begin{equation*}
	p^{(k+1)} \gets \prox[Big]{\primalstep_k \Fidelity}{ p^{(k)} - \primalstep_k \paren[big](){\Lambda^* \xi^{(k+1)}}^\sharp }
	.
\end{equation*}
Here $\sharp \colon \cX^* \to \cX$ denotes the inverse Riesz isomorphism for the space $\cX$.
Finally, the (primal) extrapolation step becomes
\begin{equation*}
	\bar p^{(k+1)} \gets p^{(k+1)} - \theta_k \, \paren[big](){p^{(k)} - p^{(k+1)}}
	=
	p^{(k+1)} + \theta_k \, \paren[big](){p^{(k+1)} - p^{(k)}}
	.
\end{equation*}
The steps above agree with~\cite[Alg.~1]{ChambollePock:2011:1} (with the roles of $F$ and $G$ reversed).

\subsection{Convergence of the Linearized Chambolle--Pock Algorithm}
\label{subsec:convergence_lRCPA}

In the following we adapt the proof of~\cite{ChambollePock:2011:1} to solve the linearized saddle-point problem~\eqref{eq:Linearized_manifold_saddle-point_reprsentation_for_PD_Chambolle-Pock}.
We restrict the discussion to the case where $\cM$ and $\cN$ are Hadamard manifolds and $\cC = \cM$ and $\cD = \cN$.
Recall that in this case we have $\LocalExponentialPreimage{n}{\cN} = \tangent{n}[\cN]$ so $g_n = G \circ \exponential{n}$ holds everywhere on $\tangent{n}[\cN]$.
Moreover, we fix $m \in \cM$ and $n \coloneqq \Lambda(m) \in \cN$ during the iteration and set the acceleration parameter $\gamma$ to zero and choose the over-relaxation parameter $\theta_k \equiv 1$ in \cref{alg:DualOverrelax-lRCP}.

Before presenting the main result of this section and motivated by the condition introduced after~\cite[eq.(2.4)]{Valkonen:2014:1}, we introduce the following constant
\begin{equation}
	\label{eq:Definition_L}
	L \coloneqq \riemanniannorm{D\Lambda(m)}[n],
\end{equation}
\ie, the operator norm of $D\Lambda(m) \colon \tangent{m}[\cM] \to \tangent{n}[\cN]$.

\begin{theorem}
	\label{thm:Convergence_of_the_linearized_CP}%
	Suppose that $\cM$ and~$\cN$ are two Hadamard manifolds.
	Let $\Fidelity\colon\cM\to\eR$, $\Prior\colon\cN\to\eR$ be proper and lsc functions, and let $\Lambda \colon \cM \to \cN$ be differentiable.
	Fix $m \in \cM$ and $n \coloneqq \Lambda(m) \in \cN$.
	Assume that $F$ is geodesically convex and that $g_n = G \circ \exp_n$ is convex on $\tangent{n}[\cN]$.
	Suppose that the linearized saddle-point problem~\eqref{eq:Linearized_manifold_saddle-point_reprsentation_for_PD_Chambolle-Pock} has a saddle-point~$\paren[big](){\widehat p,\widehat \xi_n}$.
	Choose~$\primalstep$, $\dualstep$ such that~$\primalstep\dualstep L^2<1$, with~$L$ defined in~\eqref{eq:Definition_L}, and let the iterates~$\paren[big](){\xi^{(k)}_n,p^{(k)},\bar\xi_n^{(k)}}$ be given by~\cref{alg:DualOverrelax-lRCP}.
	Suppose that there exists $K \in \N$ such that for all $k \ge K$, the following holds:
	\begin{equation}
		\label{eq:Sufficient_conditions}
		C(k)
		\coloneqq
		\frac{1}{\primalstep} \dist[big]{p^{(k)}}{\widetilde p^{(k)}}[2]
		+ \dual[big]{\bar\xi_n^{(k)}}{D\Lambda(m)[\zeta_k]}
		\geq
		0
		,
	\end{equation}
	where
	\begin{equation*}
		\widetilde p^{(k)}
		\coloneqq
		\exponential[auto]{p^{(k)}}{\parallelTransport{m}{p^{(k)}}{\paren[big](){-\primalstep D\Lambda(m)^*\paren[big][]{2 \xi^{(k)}_n - \xi_n^{(k-1)}}}}^\sharp}
		,
	\end{equation*}
	and
	\begin{equation*}
		\zeta_k
		\coloneqq
		\parallelTransport{p^{(k)}}{m}{\paren[Big](){\logarithm{p^{(k)}} p^{(k+1)} - \parallelTransport{\widetilde p^{(k)}}{p^{(k)}}{\logarithm{\widetilde p^{(k)}}} \widehat p}}
		- \logarithm{m} p^{(k+1)}
		+ \logarithm{m} \widehat p
	\end{equation*}
	holds with $\bar \xi_n^{(k)} = 2\xi_n^{(k)}-\xi_n^{(k-1)}$.
	Then the following statements are true.
	\begin{enumerate}
		\item
			\label{item:Boundedness_CP}
			The sequence~$\paren[big](){p^{(k)},\xi^{(k)}_n}$ remains bounded, \ie,
			\begin{equation}
				\label{eq:Boundedness_of_CP_iterations}
				\frac{1}{2\dualstep} \riemanniannorm[big]{\widehat \xi_n - \xi^{(k)}_n}[n]^2 + \frac{1}{2\primalstep} \dist[big]{p^{(k)}}{\widehat p}[2]
				\leq
				\frac{1}{2\dualstep} \riemanniannorm[big]{\widehat \xi_n - \xi^{(0)}_n}[n]^2 + \frac{1}{2\primalstep} \dist[big]{p^{(0)}}{\widehat p}[2]
				.
			\end{equation}
		\item
			\label{item:Convergence_to_saddle-point_CP}
			There exists a saddle-point~$(p^*,\xi_n^*)$ such that~$p^{(k)}\to p^*$ and~$\xi^{(k)}_n \to \xi_n^*$.
	\end{enumerate}
\end{theorem}

\begin{remark}
	\label{rem:Interpretation_of_C}%
	A main difference of~\cref{thm:Convergence_of_the_linearized_CP} to the Hilbert space case is the condition on~$C(k)$.
	Restricting this theorem to the setting of \cref{subsec:CP_in_Hilbert_space}, the parallel transport and the logarithmic map simplify to the identity and subtraction, respectively.
	Then
	\makeatletter
	\IfSubStr{\@currentclass}{svjour3.cls}{%
		\begin{equation*}
			\begin{aligned}
				\zeta_k
				&
				=
				p^{(k+1)} - p^{(k)} - \widehat p + \widetilde p^{(k)} - p^{(k+1)} + m + \widehat p - m
				\\
				&
				=
				\widetilde p^{(k)}-p^{(k)}
				=
				- \paren[big](){\primalstep D\Lambda(m)^*[\bar\xi^{(k)}_n]}^\sharp
			\end{aligned}
		\end{equation*}
		}{
		\begin{equation*}
			\zeta_k
			=
			p^{(k+1)} - p^{(k)} - \widehat p + \widetilde p^{(k)} - p^{(k+1)} + m + \widehat p - m
			=
			\widetilde p^{(k)}-p^{(k)}
			=
			- \paren[big](){\primalstep D\Lambda(m)^*[\bar\xi^{(k)}_n]}^\sharp
		\end{equation*}
	}
	\makeatother
	holds and hence $C(k)$ simplifies to
	\begin{equation*}
		C(k)
		=
		\primalstep \, \riemanniannorm[big]{D\Lambda(m)^*[\bar\xi_n ^{(k)}]}[\cY^*]^2
		-
		\primalstep \dual[big]{\bar\xi_n^{(k)}}{D\Lambda(m)\paren[big][]{\paren[big](){D\Lambda(m)^*[\bar\xi_n^{(k)}]}^\sharp}}
		=
		0
	\end{equation*}
	for any~$\bar \xi_n^{(k)}$, so condition~\eqref{eq:Sufficient_conditions} is satisfied for all $k\in\N$.
\end{remark}

\begin{proof}[Proof of~\cref{thm:Convergence_of_the_linearized_CP}]
	Recall that we assume~$\Lambda(m)=n$.
	Following along the lines of~\cite[Thm.~1]{ChambollePock:2011:1}, we first write a generic iteration of~\cref{alg:DualOverrelax-lRCP} for notational convenience in a general form
	\begin{equation}
		\label{eq:Thm_0_General_form_iterations}
		\begin{alignedat}{2}
			p^{(k+1)}
			&
			= \prox[big]{\primalstep F}{\widetilde p^{(k)}}
			,
			\quad
			&
			\widetilde p^{(k)}
			&
			\coloneqq
			\exponential[auto]{p^{(k)}}{\parallelTransport{m}{p^{(k)}}{\paren[auto](){-\primalstep D\Lambda(m)^*[\bar \xi_n]}}^\sharp}
			,
			\\
			\xi^{(k+1)}_n
			&
			=
			\prox[big]{\dualstep \Prior^*_n}{\widetilde \xi^{(k)}_n}
			,
			\quad
			&
			\widetilde \xi_n^{(k)}
			&
			\coloneqq
			\xi_n^{(k)} + \dualstep \paren[auto](){D\Lambda(m)[\logarithm{m} \bar p]}^\flat
			.
		\end{alignedat}
	\end{equation}
	We are going to insert $\bar p = p^{(k+1)}$ and $\bar \xi_n = 2 \xi^{(k)}_n - \xi_n^{(k-1)}$ later on, which ensure the iterations agree with \cref{alg:DualOverrelax-lRCP}.
	Applying~\cref{lem:proxSubdiff}, we get
	\begin{equation}
		\label{eq:Thm_1_Equiv_general_form_iterations}
		\begin{aligned}
			\frac{1}{\primalstep} \paren[big](){\logarithm{p^{(k+1)}} \widetilde p^{(k)}}^\flat
			&
			\in \partial_\cM F\paren[big](){p^{(k+1)}}
			,
			\\
			\dualstep^{-1} \paren[big](){\xi_n^{(k)} - \xi_n^{(k+1)}}^\sharp
			+ D\Lambda(m)[\logarithm{m} \bar p]
			&
			\in \partial G_n^*\paren[big](){\xi^{(k+1)}_n}
			.
		\end{aligned}
	\end{equation}
	Due to~\cref{def:Classiscal_convex_subdifferential} and~\cref{def:Subdifferential}, we obtain for every~$\xi_n\in \cotangent{n}[\cN]$ and~$p\in \cM$ the inequalities
	\makeatletter
	\begin{align}
		F(p)
		&
		\geq F(p^{(k+1)}) + \frac{1}{\primalstep} \riemannian[big]{\logarithm{p^{(k+1)}} \widetilde p^{(k)}}{\logarithm{p^{(k+1)}} p}[p^{(k+1)}],
		\nonumber
		\\
		G_n^* (\xi_n)
		&
		\geq
		G_n^*\paren[big](){\xi^{(k+1)}_n}
		+ \frac{1}{\dualstep} \riemannian[big]{\xi^{(k)}_n -\xi^{(k+1)}_n}{\xi_n - \xi^{(k+1)}_n}[n]
		\IfSubStr{\@currentclass}{svjour3.cls}{%
			\nonumber
			\\
			&
			\quad
			}{%
		}
		+ \dual[big]{\xi_n - \xi^{(k+1)}_n}{D\Lambda (m) [\logarithm{m} \bar p]}
		.
		\label{eq:Thm_1_Equivalent_with_subdifferential}
	\end{align}
	\makeatother
	A concrete choice for $p$ and $\xi_n$ will be made below.
	Now we consider the geodesic triangle~$\Delta = \paren[big](){\widetilde p^{(k)},p^{(k+1)},p}$.
	Applying the law of cosines in Hadamard manifolds (\cite[Thm.~2.2]{FerreiraOliveira:2002:1}), we obtain
	\makeatletter
	\IfSubStr{\@currentclass}{svjour3.cls}{%
		\begin{multline*}
			\frac{1}{\primalstep}
			\riemannian[big]{\logarithm{p^{(k+1)}} \widetilde p^{(k)}}{\logarithm{p^{(k+1)}} p}[p^{(k+1)}]
			\\
			\geq
			\frac{1}{2\primalstep} \dist[big]{\widetilde p^{(k)}}{p^{(k+1)}}[2]
			+
			\frac{1}{2\primalstep} \dist[big]{p}{p^{(k+1)}}[2]
			- \frac{1}{2\primalstep} \dist[big]{\widetilde p^{(k)}}{p}[2]
			.
		\end{multline*}
		}{%
		\begin{equation*}
			\frac{1}{\primalstep}
			\riemannian[big]{\logarithm{p^{(k+1)}} \widetilde p^{(k)}}{\logarithm{p^{(k+1)}} p}[p^{(k+1)}]
			\geq
			\frac{1}{2\primalstep} \dist[big]{\widetilde p^{(k)}}{p^{(k+1)}}[2]
			+
			\frac{1}{2\primalstep} \dist[big]{p}{p^{(k+1)}}[2]
			- \frac{1}{2\primalstep} \dist[big]{\widetilde p^{(k)}}{p}[2]
			.
		\end{equation*}
	}
	\makeatother
	Rearranging the law of cosines for the triangle~$\Delta = \paren[big](){p^{(k)},\widetilde p^{(k)},p}$ yields
	\makeatletter
	\IfSubStr{\@currentclass}{svjour3.cls}{%
		\begin{multline*}
			- \frac{1}{2\primalstep} \dist[big]{\widetilde p^{(k)}}{p}[2]
			\\
			\geq
			\frac{1}{2\primalstep} \dist[big]{\widetilde p^{(k)}}{p^{(k)}}[2]
			-
			\frac{1}{2\primalstep} \dist[big]{p^{(k)}}{p}[2]
			-
			\frac{1}{\primalstep}
			\riemannian[big]{\logarithm{\widetilde p^{(k)}} p^{(k)}}{ \logarithm{\widetilde p^{(k)}} p}[\widetilde{p}^{(k)}].
		\end{multline*}
		}{%
		\begin{equation*}
			- \frac{1}{2\primalstep} \dist[big]{\widetilde p^{(k)}}{p}[2]
			\geq
			\frac{1}{2\primalstep} \dist[big]{\widetilde p^{(k)}}{p^{(k)}}[2]
			-
			\frac{1}{2\primalstep} \dist[big]{p^{(k)}}{p}[2]
			-
			\frac{1}{\primalstep}
			\riemannian[big]{\logarithm{\widetilde p^{(k)}} p^{(k)}}{ \logarithm{\widetilde p^{(k)}} p}[\widetilde{p}^{(k)}].
		\end{equation*}
	}
	\makeatother
	We rephrase the last term as
	\begin{align*}
		\MoveEqLeft
		-\frac{1}{\primalstep} \riemannian[big]{\logarithm{\widetilde p^{(k)}} p^{(k)}}{\logarithm{\widetilde p^{(k)}} p}[\widetilde{p}^{(k)}]
		\\
		&
		=
		- \frac{1}{\primalstep}\riemannian[auto]{%
			\parallelTransport{\widetilde p^{(k)}}{p^{(k)}}{\logarithm{\widetilde p^{(k)}} p^{(k)}}
			}{
			\parallelTransport{\widetilde p^{(k)}}{p^{(k)}}{\logarithm{\widetilde p^{(k)}} p}
		}[p^{(k)}]
		\\
		&
		=
		-\frac{1}{\primalstep}\riemannian[auto]{
			-\logarithm{p^{(k)}} \widetilde p^{(k)}
			}{
			\parallelTransport{\widetilde p^{(k)}}{p^{(k)}}{\logarithm{\widetilde p^{(k)}} p}
		}[p^{(k)}]
		\\
		&
		=
		- \riemannian[auto]{
			D\Lambda(m)^*[\bar \xi_n]^\sharp
			}{
			\parallelTransport{p^{(k)}}{m}{\parallelTransport{\widetilde p^{(k)}}{p^{(k)}}{\logarithm{\widetilde p^{(k)}} p}}
		}[m]
		\\
		&
		=
		-\dual[big]{\bar \xi_n}{D\Lambda(m)[\parallelTransport{p^{(k)}}{m}{\parallelTransport{\widetilde p^{(k)}}{p^{(k)}}{\logarithm{\widetilde p^{(k)}} p}}]}
		.
	\end{align*}
	We insert the estimates above into the first inequality in~\eqref{eq:Thm_1_Equivalent_with_subdifferential} to obtain
	\begin{align*}
		F(p)
		&
		\geq
		F(p^{(k+1)})
		+ \frac{1}{2\primalstep} \dist[big]{\widetilde p^{(k)}}{p^{(k+1)}}[2]
		+ \frac{1}{2\primalstep} \dist[big]{p^{(k+1)}}{p}[2]
		+ \frac{1}{2\primalstep} \dist[big]{\widetilde p^{(k)}}{p^{(k)}}[2]
		\\
		&
		\quad
		- \frac{1}{2\primalstep} \dist[big]{p^{(k)}}{p}[2]
		- \dual[big]{%
			\bar\xi_n}{%
			D\Lambda(m)[%
		\parallelTransport{p^{(k)}}{m}{\parallelTransport{\widetilde p^{(k)}}{p^{(k)}} {\logarithm{\widetilde p^{(k)}} p}}]}
		.
	\end{align*}
	Considering now the geodesic triangle~$\Delta = \paren[big](){\widetilde p^{(k)},p^{(k)},p^{(k+1)}}$, we get
	\begin{align*}
		\frac{1}{2\primalstep} \dist[big]{p^{(k+1)}}{\widetilde p^{(k)}}[2]
		&
		\geq
		\frac{1}{2\primalstep} \dist[big]{p^{(k)}}{p^{(k+1)}}[2]
		+ \frac{1}{2\primalstep} \dist[big]{p^{(k)}}{\widetilde p^{(k)}}[2]
		\\
		&
		\quad
		- \frac{1}{\primalstep}
		\riemannian[big]{%
			\logarithm{p^{(k)}} \widetilde p^{(k)}
			}{%
			\logarithm{p^{(k)}} p^{(k+1)}
		}[p^{(k)}]
		,
	\end{align*}
	and, noticing that
	\begin{equation*}
		-\frac{1}{\primalstep}
		\riemannian[big]{\logarithm{p^{(k)}} \widetilde p^{(k)}}{\logarithm{p^{(k)}} p^{(k+1)}}[p^{(k)}]
		=
		\dual[big]{\bar\xi_n}{D\Lambda(m)[\parallelTransport{p^{(k)}}{m}{\logarithm{p^{(k)}} p^{(k+1)}}]}
	\end{equation*}
	holds, we write
	\begin{align*}
		F(p)
		&
		\geq
		F(p^{(k+1)})
		+ \frac{1}{2\primalstep} \dist[big]{p^{(k+1)}}{p}[2]
		- \frac{1}{2\primalstep} \dist[big]{p^{(k)}}{p}[2]
		+ \frac{1}{2\primalstep} \dist[big]{p^{(k)}}{p^{(k+1)}}[2]
		\\
		&
		\quad
		+ \frac{1}{\primalstep} \dist[big]{p^{(k)}}{\widetilde p^{(k)}}[2]
		\\
		&
		\quad
		+ \dual[big]{\bar\xi_n}{%
			D\Lambda(m)[
			\parallelTransport{p^{(k)}}{m}{\logarithm{p^{(k)}} p^{(k+1)}}
			- \parallelTransport{p^{(k)}}{m}{\parallelTransport{\widetilde p^{(k)}}{p^{(k)}}{\logarithm{\widetilde p^{(k)}} p}}
			]
		}
		.
	\end{align*}
	Adding this inequality with the second inequality from~\eqref{eq:Thm_1_Equivalent_with_subdifferential}, we get
	\begin{subequations}
		\label{eq:Main_ineq_in_boundedness}
		\begin{align}
			\MoveEqLeft
			\frac{1}{2\dualstep} \riemanniannorm[big]{\xi_n - \xi^{(k)}_n}[n]^2 + \frac{1}{2\primalstep} \dist[big]{p^{(k)}}{p}[2]
			\nonumber
			\\
			&
			\geq
			\dual[big]{D\Lambda (m) [\logarithm{m} p^{(k+1)}]}{\xi_n} + F(p^{(k+1)}) - G_n^*(\xi_n)
			\nonumber
			\\
			&
			\quad
			- \paren[Big][]{\dual[big]{D\Lambda(m)[\logarithm{m} p]}{\xi^{(k+1)}} + F(p) - G_n^*\paren[big](){\xi^{(k+1)}_n}}
			\nonumber
			\\
			&
			\quad
			+ \frac{1}{2\dualstep} \riemanniannorm[big]{\xi_n - \xi^{(k+1)}_n}[n]^2
			+ \frac{1}{2\dualstep} \riemanniannorm[big]{\xi^{(k)}_n - \xi^{(k+1)}_n}[n]^2
			\nonumber
			\\
			&
			\quad
			+ \frac{1}{2\primalstep} \dist[big]{p^{(k+1)}}{ p}[2]
			+ \frac{1}{2\primalstep} \dist[big]{p^{(k)}}{p^{(k+1)}}[2]
			\nonumber
			\\
			&
			\quad
			+ \frac{1}{\primalstep} \dist[big]{p^{(k)}}{\widetilde p^{(k)}}[2]
			\label{eq:Thm1_Auxiliar_0}
			\\
			&
			\quad
			+ \dual[big]{\bar\xi_n}{D\Lambda(m)[ \parallelTransport{p^{(k)}}{m}{\logarithm{p^{(k)}} p^{(k+1)}}
			- \parallelTransport{p^{(k)}}{m}{\parallelTransport{\widetilde p^{(k)}}{p^{(k)}}{\logarithm{\widetilde p^{(k)}} p}}]}
			\label{eq:Thm1_Auxiliar_a}
			\\
			&
			\quad
			+ \dual[big]{\xi^{(k+1)}_n - \xi_n}{ D\Lambda(m)[\logarithm{m} p^{(k+1)} - \logarithm{m} \bar p ]}
			\label{eq:Thm1_Auxiliar_b}
			\\
			&
			\quad
			- \dual[big]{\xi^{(k+1)}_n - \bar \xi_n}{D\Lambda(m)[\logarithm{m} p^{(k+1)} - \logarithm{m} p]}
			\label{eq:Thm1_Auxiliar_c}
			\\
			&
			\quad
			- \dual[big]{\bar \xi_n}{ D\Lambda(m)[ \logarithm{m} p^{(k+1)} - \logarithm{m} p]}
			.
			\label{eq:Thm1_Auxiliar_d}
		\end{align}
	\end{subequations}
	Recalling now the choice $\bar p = p^{(k+1)}$, the term \eqref{eq:Thm1_Auxiliar_b} vanishes.
	We also insert $\bar \xi_n = 2 \xi^{(k)}_n - \xi_n^{(k-1)}$ and estimate \eqref{eq:Thm1_Auxiliar_c} according to
	\begin{align*}
		\MoveEqLeft
		- \dual[big]{\xi_n^{(k+1)} - \bar \xi_n}{D\Lambda(m)[\logarithm{m} p^{(k+1)} - \logarithm{m} p]}
		\\
		&
		=
		- \dual[big]{%
			\xi_n^{(k+1)} - \xi_n^{(k)} - \paren[big](){\xi_n^{(k)} - \xi_n^{(k-1)}}
			}{
		D\Lambda(m)[\logarithm{m} p^{(k+1)} - \logarithm{m} p]}
		\\
		&
		=
		- \dual[big]{\xi_n^{(k+1)} - \xi_n^{(k)}}{
			D\Lambda(m)[\logarithm{m} p^{(k+1)} - \logarithm{m} p]
		}
		\\
		&
		\quad
		+ \dual[big]{\xi_n^{(k)} - \xi_n^{(k-1)}}{
			D\Lambda(m)[\logarithm{m} p^{(k)} - \logarithm{m} p]
		}
		\\
		&
		\quad
		- \dual[big]{\xi_n^{(k-1)} - \xi_n^{(k)}}{
			D\Lambda(m)[\logarithm{m} p^{(k+1)} - \logarithm{m} p^{(k)}]
		}
		\\
		&
		\geq
		- \dual[big]{\xi_n^{(k+1)} - \xi_n^{(k)}}{
			D\Lambda(m)[\logarithm{m} p^{(k+1)} - \logarithm{m} p]
		}
		\\
		&
		\quad
		+ \dual[big]{\xi_n^{(k)} - \xi_n^{(k-1)}}{
			D\Lambda(m)[\logarithm{m} p^{(k)} - \logarithm{m} p]
		}
		\\
		&
		\quad
		- L \, \riemanniannorm[big]{\xi^{(k)}_n - \xi^{(k-1)}}[n]
		\riemanniannorm[big]{\logarithm{m} p^{(k+1)} - \logarithm{m} p^{(k)}}[m]
		.
	\end{align*}
	Using that~$2ab\leq \alpha a^2+b^2/\alpha$ holds for every~$a,b \ge 0$ and $\alpha>0$, and choosing $\alpha= \frac{\sqrt{\dualstep}}{\sqrt{\primalstep}}$, we get
	\begin{align}
		\MoveEqLeft
		- \dual[big]{\xi_n^{(k+1)} - \bar \xi_n}{
			D\Lambda(m)[\logarithm{m} p^{(k+1)} - \logarithm{m} p]
		}
		\nonumber
		\\
		&
		\geq
		- \dual[big]{\xi_n^{(k+1)} - \xi_n^{(k)}}{
			D\Lambda(m)[\logarithm{m} p^{(k+1)} - \logarithm{m} p]
		}
		\nonumber
		\\
		&
		\quad
		+ \dual[big]{\xi_n^{(k)} - \xi_n^{(k-1)}}{
			D\Lambda(m)[\logarithm{m} p^{(k)} - \logarithm{m} p]
		}
		\nonumber
		\\
		&
		\quad
		- \frac{L\sqrt{\dualstep}}{2\sqrt{\primalstep}} \dist[big]{p^{(k+1)}}{p^{(k)}}[2]
		- \frac{L\sqrt{\primalstep}}{2\sqrt{\dualstep}} \riemanniannorm{\xi^{(k-1)}_n - \xi^{(k)}_n}[n]^2
		,
		\label{eq:Convergence_of_the_linearized_CP_1}
	\end{align}
	where~$L$ is the constant defined in~\eqref{eq:Definition_L}.

	We now make the choice $p = \widehat p$ and notice that the sum of~\eqref{eq:Thm1_Auxiliar_0},~\eqref{eq:Thm1_Auxiliar_a} and~\eqref{eq:Thm1_Auxiliar_d} corresponds to~$C(k)$.
	We also notice that the first two lines on the right hand side of~\eqref{eq:Convergence_of_the_linearized_CP_1} are the primal-dual gap, denoted in the following by $\PDG{k}$.
	Moreover, we set $\xi_n = \widehat \xi_n$.
	With these substitutions in~\eqref{eq:Thm1_Auxiliar_0}--\eqref{eq:Thm1_Auxiliar_d}, we arrive at the estimate
	\begin{align}
		\MoveEqLeft
		\frac{1}{2\dualstep} \riemanniannorm[big]{\widehat \xi_n - \xi^{(k)}_n}[n]^2
		+ \frac{1}{2\primalstep} \dist[big]{p^{(k)}}{\widehat p}[2]
		\nonumber
		\\
		&
		\geq
		\PDG{k} + C(k)
		\nonumber
		\\
		&
		\quad
		+ \paren[auto](){\frac{1}{2\primalstep} - \frac{L\sqrt{\dualstep}}{2\sqrt{\primalstep}}} \dist[big]{p^{(k)}}{p^{(k+1)}}[2]
		+ \frac{1}{2\primalstep} \dist[big]{p^{(k+1)}}{\widehat p}[2]
		\nonumber
		\\
		&
		\quad
		+ \frac{1}{2\dualstep} \riemanniannorm[big]{\widehat \xi_n - \xi^{(k+1)}_n}[n]^2
		+ \frac{1}{2\dualstep} \riemanniannorm[big]{\xi^{(k)}_n - \xi^{(k+1)}_n}[n]^2
		- \frac{L\sqrt{\primalstep}}{2\sqrt{\dualstep}} \riemanniannorm[big]{\xi^{(k-1)}_n - \xi^{(k)}_n}[n]^2
		\nonumber
		\\
		&
		\quad
		- \dual[big]{\xi_n^{(k+1)} - \xi_n^{(k)}}{D\Lambda(m)[\logarithm{m} p^{(k+1)} - \logarithm{m} \widehat p]}
		\nonumber
		\\
		&
		\quad
		+ \dual[big]{\xi_n^{(k)} - \xi_n^{(k-1)}}{D\Lambda(m)[\logarithm{m} p^{(k)} - \logarithm{m} \widehat p]}
		.
		\label{eq:Chain_for_convergence_2}
	\end{align}
	We continue to sum~\eqref{eq:Chain_for_convergence_2} from~$0$ to~$N-1$, where we set~$\xi_n^{(-1)} \coloneqq \xi_n ^{(0)}$ in coherence with the initial choice~$\bar \xi_n^{(0)} = \xi_n^{(0)}$.
	We obtain
	\begin{align}
		\MoveEqLeft
		\frac{1}{2\dualstep} \riemanniannorm[big]{\widehat \xi_n - \xi^{(0)}_n}[n]^2
		+ \frac{1}{2\primalstep} \dist[big]{p^{(0)}}{\widehat p}[2]
		\nonumber
		\\
		&
		\geq
		\sum_{k=0}^{N-1} \PDG{k} + \sum_{k=0}^{N-1} C(k)
		+ \frac{1}{2\dualstep}\riemanniannorm[big]{\widehat \xi_n - \xi^{(N)}_n}[n]^2
		+ \frac{1}{2\primalstep} \dist[big]{p^{(N)}}{\widehat p}[2]
		\nonumber
		\\
		&
		\quad
		+ \paren[auto](){\frac{1}{2\primalstep} - \frac{L\sqrt{\dualstep}}{2\sqrt{\primalstep}}} \sum_{k=1}^{N} \dist[big]{p^{(k)}}{p^{(k-1)}}[2]
		+ \paren[auto](){\frac{1}{2\dualstep} - \frac{L\sqrt{\primalstep}}{2\sqrt{\dualstep}}} \sum_{k=1}^{N-1} \riemanniannorm[big]{\xi^{(k)}_n - \xi^{(k-1)}_n}[n]^2
		\nonumber
		\\
		&
		\quad
		+ \frac{1}{2\dualstep} \riemanniannorm[big]{\xi_n^{(N-1)} - \xi_n^{(N)}}[n]^2
		- \dual[big]{\xi_n^{(N)} - \xi_n^{(N-1)}}{D\Lambda(m)[\logarithm{m} p^{(N)} - \logarithm{m} \widehat p]}
		.
		\label{eq:Chain_for_convergence_3}
	\end{align}
	We further develop the last term in~\eqref{eq:Chain_for_convergence_3} and get
	\begin{align*}
		\MoveEqLeft
		- \dual[big]{\xi_n^{(N)} - \xi_n^{(N-1)}}{D\Lambda(m)[\logarithm{m} p^{(N)} - \logarithm{m} \widehat p]}
		\\
		&
		\geq
		- L \, \riemanniannorm[big]{\xi_n^{(N)} - \xi_n^{(N-1)}}[n] \dist[big]{p^{(N)}}{\widehat p}
		\\
		&
		\geq
		- \frac{L\alpha}{2}\riemanniannorm[big]{\xi_n^{(N)} - \xi_n^{(N-1)}}[n]^2 - \frac{L}{2\alpha} \dist[big]{p^{(N)}}{\widehat p}[2]
		.
	\end{align*}
	Choosing $\alpha = 1/(\dualstep L)$, we conclude
	\begin{align*}
		\MoveEqLeft
		- \dual[big]{\xi_n^{(N)} - \xi_n^{(N-1)}}{D\Lambda(m)[\logarithm{m} p^{(N)} - \logarithm{m} \widehat p]}
		\\
		&
		\geq
		- \frac{1}{2\dualstep}\riemanniannorm[big]{\xi_n^{(N)} - \xi_n^{(N-1)}}[n]^2 - \frac{\dualstep L^2}{2} \dist[big]{p^{(N)}}{\widehat p}[2]
		.
	\end{align*}
	Hence \eqref{eq:Chain_for_convergence_3} becomes
	\begin{align}
		\MoveEqLeft
		\frac{1}{2\dualstep} \riemanniannorm[big]{\widehat \xi_n - \xi^{(0)}_n}[n]^2
		+ \frac{1}{2\primalstep} \dist[big]{p^{(0)}}{\widehat p}[2]
		\nonumber
		\\
		&
		\geq
		\sum_{k=0}^{N-1} \PDG{k} + \sum_{k=0}^{N-1} C(k)
		\nonumber
		\\
		&
		\quad
		+ \frac{1}{2\dualstep} \riemanniannorm[big]{\widehat \xi_n - \xi^{(N)}_n}[n]^2
		+ \paren[auto](){\frac{1}{2\dualstep} - \frac{L\sqrt{\primalstep}}{2\sqrt{\dualstep}}} \sum_{k=1}^{N-1} \riemanniannorm[big]{\xi^{(k)}_n - \xi^{(k-1)}_n}[n]^2
		\nonumber
		\\
		&
		\quad
		+ \paren[auto](){\frac{1}{2\primalstep} - \frac{\dualstep L^2}{2}} \dist[big]{p^{(N)}}{\widehat p}[2]
		+ \paren[auto](){\frac{1}{2\primalstep} - \frac{L\sqrt{\dualstep}}{2\sqrt{\primalstep}}} \sum_{k=1}^{N} \dist[big]{p^{(k)}}{p^{(k-1)}}[2]
		.
		\label{eq:Chain_for_convergence_4}
	\end{align}
	Since~$\paren[big](){\widehat p ,\widehat \xi_n}$ is a saddle-point, the primal-dual gap $\PDG{k}$ is non-negative.
	Moreover, assumption~\eqref{eq:Sufficient_conditions} and the inequality~$\primalstep \dualstep L^2 < 1$ imply that the sequence~$\paren[big]\{\}{\paren[big](){p^{(k)},\xi_n^{(k)}}}$ is bounded, which is the statement~\cref{item:Boundedness_CP}.

	Part~\cref{item:Convergence_to_saddle-point_CP} follows completely analogously to the steps of~\cite[Thm.~1(c)]{ChambollePock:2011:1} adapted to~\eqref{eq:Chain_for_convergence_2}.
\end{proof}

\section{ROF Models on Manifolds}
\label{sec:ROF_Models}

A starting point of the work of~\cite{ChambollePock:2011:1} is the ROF $\ell^2$-TV denoising model~\cite{RudinOsherFatemi:1992:1}, which was generalized to manifolds in~\cite{LellmannStrekalovskiyKoetterCremers:2013:1} for the so-called isotropic and anisotropic cases.
This class of $\ell^2$-TV models can be formulated in the discrete setting as follows:
let $F = (f_{i,j})_{i,j}\in \cM^{d_1\times d_2}$, $d_1,d_2\in\N$ be a manifold-valued image, \ie, each pixel~$f_{i,j}$ takes values on a manifold~$\cM$.
Then the manifold-valued~$\ell^2$-TV energy functional reads as follows:
\begin{equation}
	\label{eq:M-LqTV}
	\cE_q(P)
	\coloneqq
	\frac{1}{2 \alpha}
	\sum_{i,j=1}^{d_1,d_2} \dist{f_{i,j}}{p_{i,j}}[2]
	+ \riemanniannorm{\nabla P}[g,q,1],
	\quad P = (p_{i,j})_{i,j} \in \cM^{d_1\times d_2}
	,
\end{equation}
where~$q \in \{1,2\}$.
The parameter~$\alpha > 0$ balances the relative influence of the data fidelity and the total varation terms in \eqref{eq:M-LqTV}.
Moreover, $\nabla \colon \cM^{d_1 \times d_2} \to \tangentBundle^{d_1 \times d_2 \times 2}$ denotes the generalization of the one-sided finite difference operator, which is defined as
\begin{equation}
	\label{eq:Nabla_operator_for_TV}
	(\nabla P)_{i,j,k}
	=
	\begin{cases}
		0 \in \tangent{p_{i,j}} & \text{ if } i = d_1 \text{ and } k = 1,
		\\
		0 \in \tangent{p_{i,j}} & \text{ if } j = d_2 \text{ and } k = 2,
		\\
		\logarithm{p_{i,j}}{p_{i+1,j}} & \text{ if } i < d_1 \text{ and } k = 1,
		\\
		\logarithm{p_{i,j}}{p_{i,j+1}} & \text{ if } j < d_2 \text{ and } k = 2.
	\end{cases}
\end{equation}
The corresponding norm in \eqref{eq:M-LqTV} is then given by
\begin{equation}
	\label{eq:General_prior_for_LqTV}
	\riemanniannorm{\nabla P}[g,q,1]
	=
	\sum_{i,j=1}^{d_1,d_2} \paren[auto](){\riemanniannorm{(\nabla P)_{i,j,1} }[g]^q + \riemanniannorm{(\nabla P)_{i,j,2} }[g]^q}^{\frac{1}{q}}
	.
\end{equation}

For simplicity of notation we do not explicitly state the base point in the Riemannian metric but denote the norm on~$\tangentBundle$ by $\riemanniannorm{\cdot}[g]$.
Depending on the value of~$q\in\{1,2\}$, we call the energy functional~\eqref{eq:M-LqTV} \emph{isotropic} when~$q=2$ and \emph{anisotropic} for~$q=1$.
Note that previous algorithms like CPPA from~\cite{WeinmannDemaretStorath:2014:1} or Douglas--Rachford~(DR) from~\cite{BergmannPerschSteidl:2016:1} are only able to tackle the anisotropic case $q=1$ due to a missing closed form of the proximal map for the isotropic TV summands.
A relaxed version of the isotropic case can be computed using the half-quadratic minimization from~\cite{BergmannChanHielscherPerschSteidl:2016:1}.
Looking at the optimality conditions of the isotropic or anisotropic energy functional, the authors in~\cite{BergmannTenbrinck:2018:1} derived and solved the corresponding $q$-Laplace equation.
This can be generalized even to all cases $q > 0$.

The minimization of~\eqref{eq:M-LqTV} fits into the setting of the model problem~\eqref{eq:General_problem_for_PD_Chambolle-Pock}.
Indeed, $\cM$ is replaced by $\cM^{d_1\times d_2}$, $\cN = \tangentBundle^{d_1\times d_2\times 2}$, $\Fidelity$ is given by the first term in~\eqref{eq:M-LqTV}, and we set $\Lambda = \nabla$ and $\Prior_{q} = \riemanniannorm{\cdot}[g,q,1]$.
The data fidelity term~$F$ clearly fulfills the assumptions stated in the beginning of \cref{sec:Primal-Dual_on_manifolds}, since the squared Riemannian distance function is geodesically convex on any strongly convex set $\cC\subset\cM$.
In particular, when $\cM$ is a Hadamard manifold, then $F$ is geodesically convex on all of~$\cM$.

While the properness and continuity of the pullback $g_n(Y) = G(\exponential{n}Y)$ are obvious, its convexity is investigated in the following.
\begin{proposition}\label{prop:conv:tn-norm}
	Suppose that $\cM$ is a Hadamard manifold and $d_1, d_2 \in \N$.
	Consider $\cM^{d_1\times d_2}$ and $\cN = \tangentBundle^{d_1 \times d_2 \times 2}$ and $G = \riemanniannorm{\cdot}[g,q,1]$ with $q \in [1,\infty)$.
	For arbitrary $n \in \cN$, define the pullback $g_{n}\colon\tangent{n}[\cN] \to \R$ by $g_n(Y) = G(\exp_nY)$.
	Then $g_n$ is a convex function on $\tangent{n}[\cN]$.
\end{proposition}
\begin{proof}
	Notice first that, since $\cM$ is Hadamard, $\cM^{d_1\times d_2}$ and $\cN$ are Hadamard as well.
	Consequently, $g_n$ is defined on all of $\tangent{n}[\cN]$.
	We are using the index $\cdot_p$ to denote points in $\cM^{d_1\times d_2}$ and the index $\cdot_X$ to denote tangent vectors.
	In particular, we denote the base point as $n=(n_p,n_X)\in\cN$.
	Let $Y=(Y_p,Y_X), Z=(Z_p,Z_X)\in\tangent{n}[\cN]$ and $t \in [0,1]$.
	Finally, we set $n' = (n'_p,n'_X) = \exponential{n}((1-t)Y + tZ)$.
	Notice that in view of the properties of the double tangent bundle as a Riemannian manifold, we have
	\begin{equation*}
		n' = \paren[auto](){n'_p, \parallelTransport{n_p}{n'_p}(n_X + (1-t) Y_X + t \, Z_X)}.
	\end{equation*}
	Therefore we obtain
	\begin{align*}
		\MoveEqLeft
		g_n((1-t)Y + tZ)
		\\
		&
		=
		G\paren[auto](){\paren[big](){n'_p, \parallelTransport{n_p}{n'_p}(n_X + (1-t) Y_X + t \, Z_X)}}
		&
		&
		\text{by definition of $g_n$}
		\\
		&
		=
		\norm{\parallelTransport{n_p}{n'_p}((1-t) (n_X + Y_X) + t \, (n_X + Z_X))}_{g,q,1}
		&
		&
		\text{by definition of $G$}
		\\
		&
		\le
		(1-t) \norm{\parallelTransport{n_p}{n'_p}(n_X + Y_X)}_{g,q,1}
		\\
		&
		\quad
		+ t \, \norm{\parallelTransport{n_p}{n'_p}(n_X + Z_X)}_{g,q,1}
		&
		&
		\text{by convexity of $G$}
		.
	\end{align*}
	Exploiting that parallel transport is an isometry, we transport the term inside the first norm to $n''_p = \exponential{n_p}Y_p$ and the term inside the second norm to $n'''_p = \exponential{n_p}Z_p$ to obtain
	\begin{align*}
		\MoveEqLeft
		g_n((1-t)Y + tZ)
		\\
		&
		\le
		(1-t) \norm{\parallelTransport{n_p}{n''_p}(n_X + Y_X)}_{g,q,1}
		+ t \, \norm{\parallelTransport{n_p}{n'''_p}(n_X + Z_X)}_{g,q,1}
		\\
		&
		=
		(1-t) G\paren[auto](){\paren[big](){n''_p, \parallelTransport{n_p}{n''_p}(n_X + Y_X)}}
		+ t \, G\paren[auto](){\paren[big](){n'''_p, \parallelTransport{n_p}{n'''_p}(n_X + Z_X)}}
		\\
		&
		=
		(1-t) \, g_n(Y) + t \, g_n(Z)
		.
	\end{align*}
\end{proof}
We apply~\cref{alg:DualOverrelax-lRCP} to solve the linearized saddle-point problem~\eqref{eq:Linearized_manifold_saddle-point_reprsentation_for_PD_Chambolle-Pock}.
This procedure will yield an approximate minimizer of~\eqref{eq:M-LqTV}.
To this end we require both the Fenchel conjugate and the proximal map of $\Prior$.
Its Fenchel dual can be stated using the dual norms, \ie, $\riemanniannorm{\cdot}[g,q^*,\infty]$ similar to Thm.~2 of~\cite{DuranMoellerSbertCremers:2016:1}, where~$q^*\in\R$ is the dual exponent of~$q$.
Let
\begin{equation*}
	B_{q^*}
	\coloneqq
	\setDef[auto]{X}{\riemanniannorm{X}[g,q^*,\infty] \leq 1}
\end{equation*}
denote the $1$-norm ball of the dual norm and
\begin{equation*}
	\iota_{B}(x)
	\coloneqq
	\begin{cases}
		0 & \text{if } x \in B,
		\\
		\infty & \text{otherwise},
	\end{cases}
\end{equation*}
the indicator function of the set $B$.
Then the Fenchel dual functions in the two cases of our main interest ($q = 1$ and $q = 2$) are
\begin{align*}
	\Prior_{2}^*(\Xi)
	=
	\iota_{B_{2}}(\Xi)
	\quad \text{and} \quad
	\Prior_{\infty}^*(\Xi)
	=
	\iota_{B_{\infty}}(\Xi).
\end{align*}
The corresponding proximal maps read as follows:
\begin{align*}
	\prox{\dualstep\Prior^*_{2}}{\Xi}
	&
	=
	\paren[Big](){{\max \paren[auto]\{\}{1, \norm[big]{\riemanniannorm{\Xi_{i,j,:}}[g]}_2} }^{-1} \Xi_{i,j,k}}_{i,j,k}
	\\
	\text{and}
	\quad
	\prox{\dualstep\Prior^*_{\infty}}{\Xi}
	&
	=
	\paren[Big](){{\max \paren[auto]\{\}{1, \riemanniannorm{\Xi_{i,j,k}}[g]}}^{-1}\Xi_{i,j,k} }_{i,j,k}
	.
\end{align*}

Finally, to derive the adjoint of $D\Lambda(m)$, let~$P\in\cM^{d_1\times d_2}$ and~$X \in \tangent{P}^{d_1\times d_2}$.
Applying the chain rule, it is not difficult to prove that
\begin{equation}
	\label{eq:Differential_nabla}
	\paren[big](){D\nabla (P)[X]}_{i,j,k}
	=
	D_1 \logarithm{p_{i,j}}{p_{i,j+e_k}}[X_{i,j}] + D_2 \logarithm{p_{i,j}}{p_{i,j + e_k}}[X_{i,j+e_k}]
\end{equation}
with the obvious modifications at the boundary.
In the above formula, $e_k$ represents either the vector $(0,1)$ or $(1,0)$ used to reach either the neighbor to the right ($k = 1$) or below ($k = 2$).
The symbols $D_1$ and $D_2$ represent the differentiation of the logarithmic map w.r.t.\ the base point and its argument, respectively.
We notice that~$D_1 \logarithm{\,\cdot\,}{p_{i,j +e_k}}$ and~$D_2 \logarithm{p_{i,j}}{\,\cdot\,}$ can be computed by an application of Jacobi fields; see for example~\cite[Lem.~4.1~(ii)~and~(iii)]{BergmannFitschenPerschSteidl:2018:1}.

With $(D\nabla)(\,\cdot\,)[\,\cdot\,] \colon \tangentBundle^{d_1\times d_2} \to \tangentBundle[\cN]$ given by Jacobi fields, its adjoint can be computed using the so-called adjoint Jacobi fields, see \eg, \cite[Sect.~4.2]{BergmannGousenbourger:2018:2}.
Defining~$N_{i,j}$ to be the set of neighbors of the pixel~$p_{i,j}$, for every~$X\in\tangent{P}^{d_1\times d_2}$ and $\eta\in\cotangent{\nabla P}[\cN]$ we have
\begin{align*}
	\MoveEqLeft
	\dual[big]{D\nabla(P)[X]}{\eta}
	\\
	&
	=
	\sum_{i,j,k} \dual[big]{(D\nabla (P)[X])_{i,j,k}}{\eta_{i,j,k}}
	\\
	&
	=
	\sum_{i,j} \sum_k \dual[big]{D_1\logarithm{p_{i,j}}{p_{i,j+e_k}}[X_{i,j}]}{\eta_{i,j,k}} + \sum_k \dual[big]{D_2\logarithm{p_{i,j}}{p_{i,j+e_k}}[X_{i,j+e_k}]}{\eta_{i,j,k}}
	\\
	&
	=
	\sum_{i,j} \sum_k \dual[big]{X_{i,j}}{D_1^*\logarithm{p_{i,j}}{p_{i,j+e_k}}[\eta_{i,j,k}]} + \sum_k \dual[big]{X_{i,j+e_k}}{D_2^*\logarithm{p_{i,j}}{p_{i,j+e_k}}[\eta_{i,j,k}]}
	\\
	&
	=
	\sum_{i,j} \dual[Big]{X_{i,j}}{\sum_k D_1^*\logarithm{p_{i,j}}{p_{i,j+e_k}} [\eta_{i,j,k}] + \sum_{(i^\prime,j^\prime)\in N_{i,j}} D_2^*\logarithm{p_{i^\prime j^\prime}}{p_{i,j}} [\eta_{i^\prime j^\prime k}]}
	\\
	&
	=
	\sum_{i,j} \dual[big]{X_{i,j}}{(D^*\nabla(P)[\eta])_{i,j}}
	,
\end{align*}
which leads to the component-wise entries in the linearized adjoint
\begin{equation}
	\label{eq:Definition_of_the_linearized_adjoint}
	\paren[big](){D^*\nabla(P)[\eta]}_{i,j}
	=
	\sum_k D_1^*\logarithm{p_{i,j}}{p_{i,j+e_k}} [\eta_{i,j,k}]
	+
	\sum_{(i^\prime,j^\prime)\in N_{i,j}} D_2^*\logarithm{p_{i^\prime j^\prime}}{p_{i,j}} [\eta_{i^\prime j^\prime k}]
	.
\end{equation}
We mention that $D_1^*\logarithm{\,\cdot\,}{p_{i,j+e_k}}$ and $D_2^* \logarithm{p_{i,j}}{\,\cdot\,}$ can also be found in~\cite[Sect.~4]{BergmannFitschenPerschSteidl:2018:1}.

\section{Numerical Experiments}
\label{sec:Numerical_experiments}

The numerical experiments are implemented in the toolbox \manoptjl\footnote{Available at \url{http://www.manoptjl.org}, following the same philosophy as the \matlab version available at \url{https://manopt.org}, see also~\cite{BoumalMishraAbsilSepulchre:2014:1}.} (\cite{Bergmann:2019:1}) in Julia\footnote{\url{https://julialang.org}}.
They were run on a MacBook~Pro, 2.5~Ghz Intel Core~i7, 16~GB RAM, with Julia~1.1.
All our examples are based on the linearized saddle-point formulation~\eqref{eq:Linearized_manifold_saddle-point_reprsentation_for_PD_Chambolle-Pock} for $\ell^2$-TV, solved with \cref{alg:DualOverrelax-lRCP}.

\subsection{A Signal with Known Minimizer}

The first example uses signal data $\cM^{d_1}$ instead of an image, where the data space is $\cM = \S^2$, the two-dimensional sphere with the round sphere Riemannian metric.
This gives us the opportunity to consider the same problem also on the embedding manifold ${(\R^3)}^{d_1}$ in order to illustrate the difference between the manifold-valued and Euclidean settings.
We construct the data ${(f_i)}_i$ such that the unique minimizer of~\eqref{eq:M-LqTV} is known in closed form.
Therefore a second purpose of this problem is to compare the numerical solution obtained by \cref{alg:DualOverrelax-lRCP}, \ie, an approximate saddle-point of the \emph{linearized} problem~\eqref{eq:Linearized_manifold_saddle-point_reprsentation_for_PD_Chambolle-Pock}, to the solution of the original saddle-point problem~\eqref{eq:Manifold_saddle-point_representation_for_PD_Chambolle-Pock}.
Third, we wish to explore how the value~$C(k)$ from~\eqref{eq:Sufficient_conditions} behaves numerically.

The piecewise constant signal is given by
\begin{equation*}
	f \in \cM^{30},
	\quad
	f_i =
	\begin{cases}
		p_1 & \text{if } i \leq 15,
		\\
		p_2 & \text{if } i > 15,
	\end{cases}
\end{equation*}
for two values~$p_1,p_2\in\cM$ specified below.

Further, since $d_2 = 1$, the isotropic and anisotropic models~\eqref{eq:M-LqTV} coincide.
The exact minimizer $\widehat p$ of~\eqref{eq:M-LqTV} is piecewise constant with the same structure as the data~$f$.
Its values are $\widehat p_1 = \geodesic<a>{p_1}{p_2}(\delta)$ and $\widehat p_2 = \geodesic<a>{p_2}{p_1}(\delta)$ where $\delta = \min \paren[big]\{\}{\frac{\alpha}{15 \dist{p_1}{p_2}}, \frac{1}{2}}$.
Notice that the notion of geodesics are different for both manifolds under consideration, and thus the exact minimizers~$\widehat p_{\R^3}$ and $\widehat p_{\S^2}$ are different.

In the following we use $\alpha = 5$ and $p_1 = \frac{1}{\sqrt{2}}(1,1,0)^\transp$ and $p_2 = \frac{1}{\sqrt{2}}(1,-1,0)^\transp$.
The data~$f$ is shown in~\cref{subfig:TVSignal:R3Orig}.
\begin{figure}[htb]\centering
	\begin{subfigure}{.49\textwidth}\centering
		\includegraphics[width=0.98\textwidth]{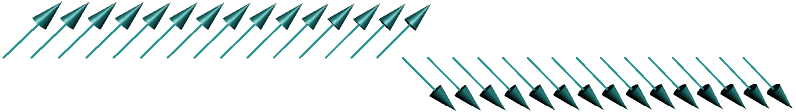}
		\caption{Signal~$f$ of unit vectors.}
		\label{subfig:TVSignal:R3Orig}
	\end{subfigure}
	\\
	\begin{subfigure}{.49\textwidth}\centering
		\includegraphics[width=0.98\textwidth]{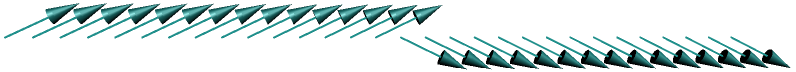}
		\caption{Minimizer with values in $\cM=\S^2$.}
		\label{subfig:TVSignal:S2Minimizer}
	\end{subfigure}
	\begin{subfigure}{.49\textwidth}\centering
		\includegraphics[width=0.98\textwidth]{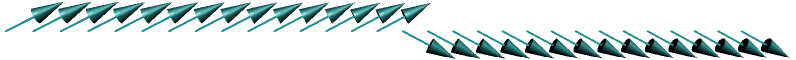}
		\caption{Minimizer with values in $\cM=\R^3$.}
		\label{subfig:TVSignal:R3Minimizer}
	\end{subfigure}
	\\
	\begin{subfigure}{.49\textwidth}\centering
		\includegraphics[width=0.98\textwidth]{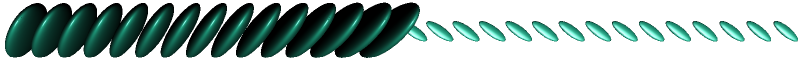}
		\caption{Signal of $\cP_+(3)$ matrices.}
		\label{subfig:TVSignal:SPDOrig}
	\end{subfigure}
	\begin{subfigure}{.49\textwidth}\centering
		\includegraphics[width=0.98\textwidth]{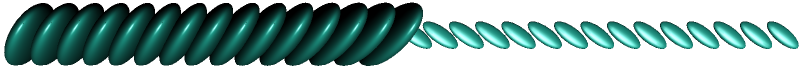}
		\caption{Minimizer on $\cM=\cP_+(3)$.}
		\label{subfig:TVSignal:SPDMinizer}
	\end{subfigure}
	\caption{Computing the minimizer of the manifold-valued $\ell^2$-TV model for a signal of unit vectors shown in~(\subref{subfig:TVSignal:R3Orig}) with respect to both manifolds $\R^3$ and $\S^2$ with~$\alpha=5$: (\subref{subfig:TVSignal:S2Minimizer}) on~$(\S^2)^{30}$ and~(\subref{subfig:TVSignal:R3Minimizer}) on~$(\R^3)^{30}$. The known effect, loss of contrast is different for both cases, since on $\S^2$ the vector remain of unit length. The same effect can be seen for a signal of spd matrices, \ie, $\cP_+(3)$; see~(\subref{subfig:TVSignal:SPDOrig}) and~(\subref{subfig:TVSignal:SPDMinizer}).}
	\label{fig:S2R3Signal}
\end{figure}

We applied the linearized Riemannian Chambolle--Pock \cref{alg:DualOverrelax-lRCP} with relaxation parameter~$\theta = 1$ on the dual variable as well as~$\primalstep = \dualstep = \frac{1}{2}$, and $\gamma = 0$, \ie, without acceleration, as well as initial guesses~$p^{(0)} = f$ and $\xi^{(0)}_n$ as the zero vector.
The stopping criterion was set to~$500$ iterations to compare run times on different manifolds. As linearization point~$m$ we use the mean of the data, which is just~$m = \geodesic<a>{p_1}{p_2}(\frac{1}{2})$.
We further set $n = \Lambda(m)$ for the base point of the Fenchel dual of $G$.
For the Euclidean case $\cM = \R^3$, we obtain a shifted version of the original Chambolle--Pock algorithm, since $m \neq 0$.

While the algorithm on~$\cM = \S^2$ takes about~$0.85$~seconds, the Euclidean algorithm takes about~$0.44$~seconds for the same number of iterations, which is most likely due to the exponential and logarithmic maps as well as the parallel transport on~$\S^2$, which involve sines and cosines.
The results obtained by the Euclidean algorithm is~$2.18\cdot 10^{-12}$ away in terms of the Euclidean norm from the analytical minimizer~$\widehat p_{\R^3}$.
Notice that the convergence of the Euclidean algorithm is covered by the theory in~\cite{ChambollePock:2011:1}.
Moreover, notice that in this setting, $\Lambda$ is a linear map between vector spaces.
During the iterations, we confirmed that the value of~$C(k)$ is numerically zero (within $\pm5.55\cdot10^{-17}$), as expected from \cref{rem:Interpretation_of_C}.

Although \cref{alg:DualOverrelax-lRCP} on $\cM = \S^2$ is based on the \emph{linearized} saddle-point problem~\eqref{eq:Linearized_manifold_saddle-point_reprsentation_for_PD_Chambolle-Pock} instead of~\eqref{eq:Manifold_saddle-point_representation_for_PD_Chambolle-Pock}, we observed that it converges to the exact minimizer~$\widehat p_{\S^2}$ of~\eqref{eq:M-LqTV}.
Therefore it is meaningful to plug in $\widehat p_{\S^2}$ into the formula~\eqref{eq:Sufficient_conditions} to evaluate $C(k)$ numerically.
The numerical values observed throughout the 500~iterations are in the interval $[-4.0\cdot10^{-13}, 4.0\cdot10^{-9}]$.
We interpret this as confirmation that $C(k)$ is non-negative in this case.
However, even with this observation the convergence of \cref{alg:DualOverrelax-lRCP} is not covered by \cref{thm:Convergence_of_the_linearized_CP} since $\S^2$ is not a Hadamard manifold.
Quite to the contrary, it has constant \emph{positive} sectional curvature.

The results are shown in~\cref{subfig:TVSignal:S2Minimizer} and~\cref{subfig:TVSignal:R3Minimizer}, respectively.
They illustrate the capability for preservation of edges, yet also a loss of contrast and reduction of jump heights well known for $\ell^2$-TV problems.
This leads to shorter vectors in $\widehat p_{\R^3}$, while, of course, their unit length is preserved in $\widehat p_{\S^2}$.

We also constructed a similar signal on~$\cM = \cP_+(3)$, the manifold of symmetric positive definite (SPD) matrices with affine-invariant metric; see~\cite{PennecFillardAyache:2006:1}.
This is a Hadamard manifold with non-constant curvature.
Let $I\in\R^{3\times3}$ denote the unit matrix and
\makeatletter
\IfSubStr{\@currentclass}{svjour3.cls}{%
	\begin{equation*}
		p_1 = \exponential[auto]{I}{\frac{2}{\riemanniannorm{X}[I]}X},
		\quad
		p_2 = \exponential[auto]{I}{-\frac{2}{\riemanniannorm{X}[I]}X}
	\end{equation*}
	with
	$X = \frac{1}{2} \begin{psmallmatrix}1&2&2\\2&2&0\\2&0&6\end{psmallmatrix} \in \cT_I\cP_+(3)$.
	}{%
	\begin{equation*}
		p_1 = \exponential[auto]{I}{\frac{2}{\riemanniannorm{X}[I]}X},
		\quad
		p_2 = \exponential[auto]{I}{-\frac{2}{\riemanniannorm{X}[I]}X}
		\quad
		\text{with }
		X = \frac{1}{2} \begin{pmatrix}1&2&2\\2&2&0\\2&0&6\end{pmatrix}
		\in \cT_I\cP_+(3).
	\end{equation*}
}
\makeatother
In this case, the run time is~$5.94$~seconds, which is due to matrix exponentials and logarithms as well as singular value decompositions that need to be computed.
Here, $C(k)$ turns out to be numerically zero (within $\pm8\cdot10^{-15}$) and the distance to the analytical minimizer~$\widehat p_{\cP_+(3)}$ is~$1.08\cdot 10^{-12}$.
The original data~$f$ and the result~$\widehat p_{\cP_+(3)}$ (again with a loss of contrast as expected) are shown in~\cref{subfig:TVSignal:SPDOrig} and~\cref{subfig:TVSignal:SPDMinizer}, respectively.

\subsection{A Comparison of Algorithms}

As a second example we compare \cref{alg:DualOverrelax-lRCP} to the cyclic proximal point algorithm (CPPA) from~\cite{Bacak:2014:1}, which was first applied to $\ell^2$-TV problems in~\cite{WeinmannDemaretStorath:2014:1}.
It is known to be a robust but generally slow method.
We also compare the proposed method with the parallel Douglas--Rachford algorithm (PDRA), which was introduced in~\cite{BergmannPerschSteidl:2016:1}.

As an example, we use the anisotropic $\ell^2$-TV model, \ie, \eqref{eq:M-LqTV} with $q = 1$, on images of size $32 \times 32$ with values in the manifold of $3 \times 3$~SPD matrices $\cP_+(3)$ as in the previous subsection.
The original data is shown in \cref{subfig:SPDImg:Orig}.
No exact solution is known for this example.
We use a regularization parameter of $\alpha = 6$.
To generate a reference solution we allowed the CPPA with step size $\lambda_k = \frac{4}{k}$ to run for~$\num{4000}$~iterations.
This required $1235.18$~seconds and it yields a value of the objective function~\eqref{eq:M-LqTV} of approximately~$38.7370$, see the bottom gray line in \cref{subfig:SPDImg:Cost}.
The result is shown in~\cref{subfig:SPDImg:Result}.

We compare CPPA to PDRA as well as to our \cref{alg:DualOverrelax-lRCP}, using the value of the cost function and the run time as criteria.
The PDRA was run with parameters~$\eta = 0.58$, $\lambda = 0.93$, which where used by~\cite{BergmannPerschSteidl:2016:1} for a similar example.
It took $379.7$~seconds to perform $122$~iterations in order to reach the same value of the cost function as obtained by CPPA.
The main bottleneck is the approximate evaluation of the involved mean, which has to be computed in every iteration.
Here we performed 20~gradient descent steps for this purpose.

For \cref{alg:DualOverrelax-lRCP} we set~$\primalstep = \dualstep = 0.4$ and~$\gamma = 0.2$.
We choose the base point $m\in\cP_+(3)^{32\times 32}$ to be the constant image of unit matrices so that $n=\Lambda(m)$ consists of zero matrices. We initialize the algorithm with $p^{(0)} = f$ and $\xi^{(0)}_n$ as the zero vector.
Our algorithm stops after $113$ iterations, which take~$96.20$~seconds, when the value of~\eqref{eq:M-LqTV} was below the value obtained by the CPPA.
While the CPPA requires about half a second per iteration, our method requires a little less than a second per iteration, but it also requires only a fraction of the iteration count of CPPA.
The behavior of the cost function is shown in~\cref{subfig:SPDImg:Cost}, where the horizontal axis (iteration number) is shown in log scale, since the \eqq{tail} of CPPA is quite long.

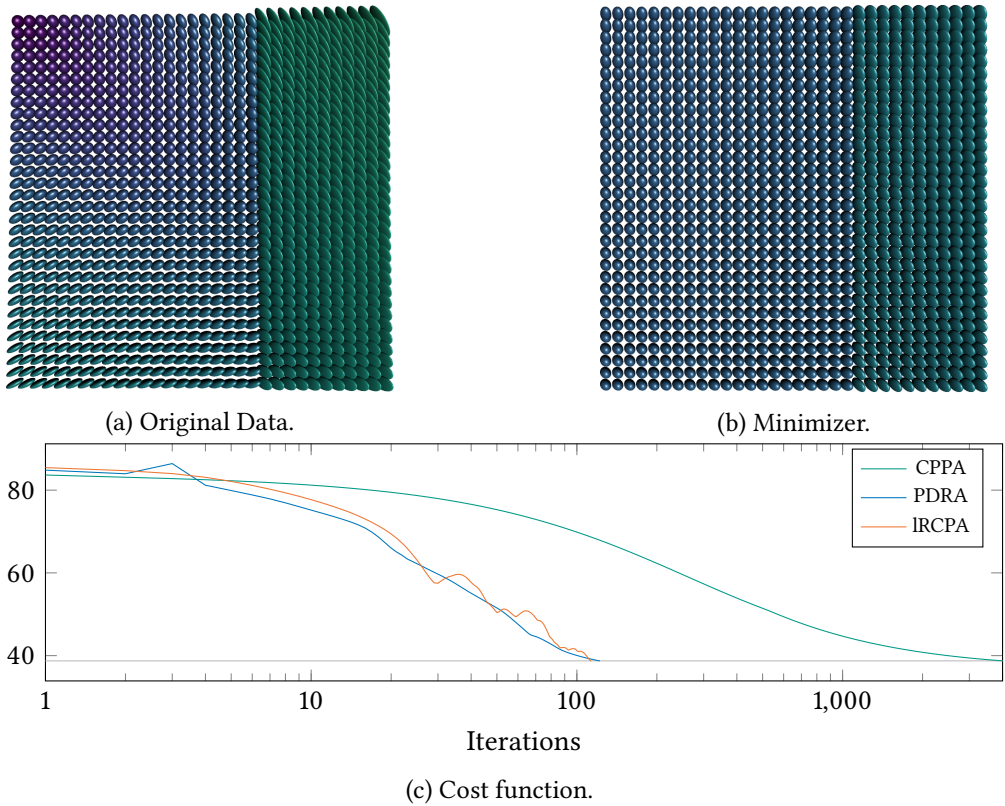
\begin{figure}[htb]\centering
	\ifthenelse{\boolean{withpictures}}{%
		\begin{subfigure}{.49\textwidth}\centering
			\includegraphics[width=0.66\textwidth]{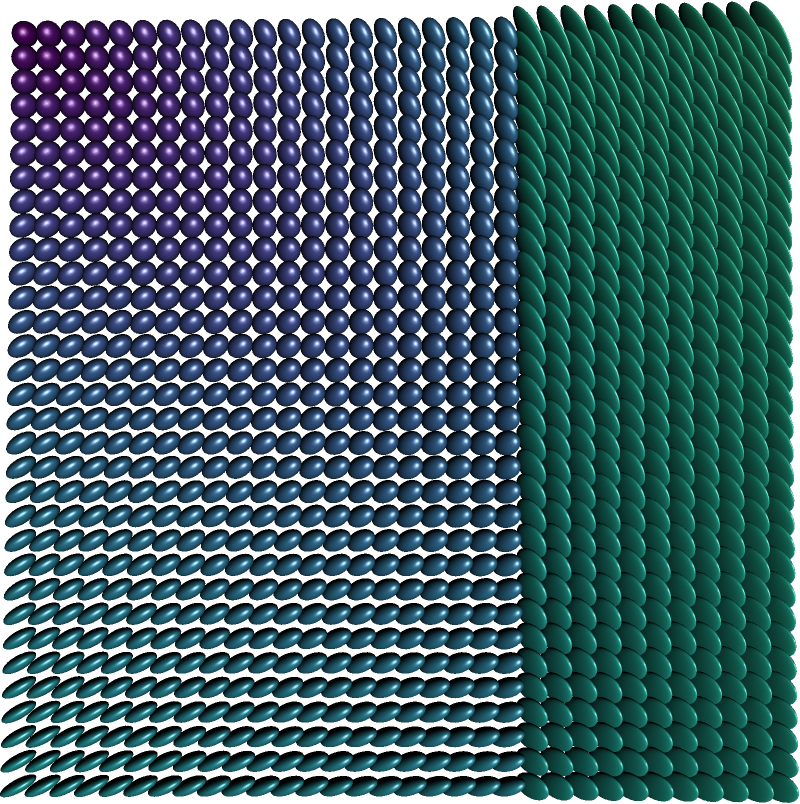}
			\caption{Original Data.}
			\label{subfig:SPDImg:Orig}
		\end{subfigure}
		\begin{subfigure}{.49\textwidth}\centering
			\includegraphics[width=0.66\textwidth]{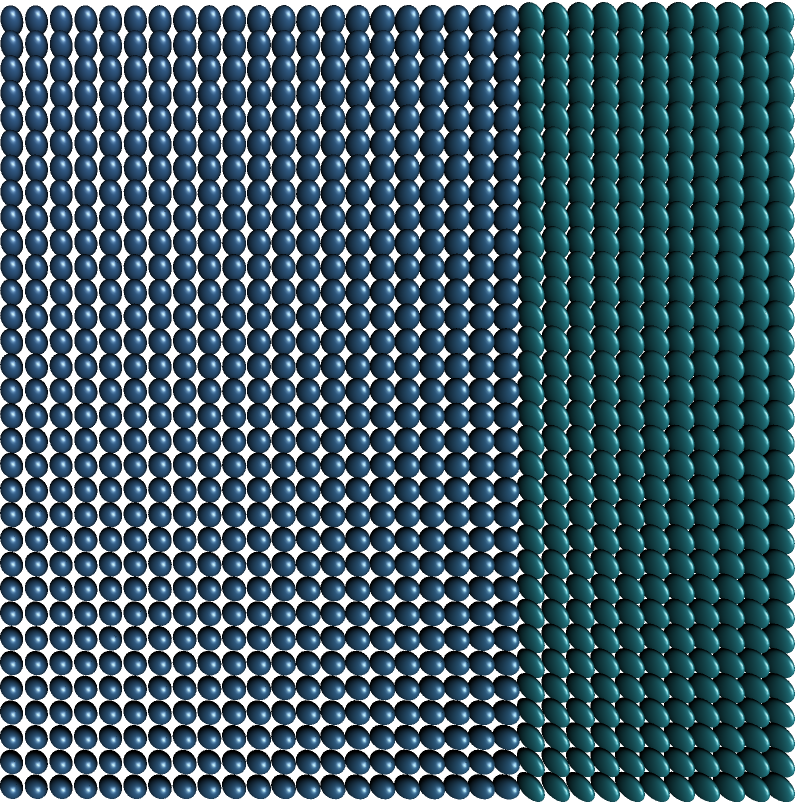}
			\caption{Minimizer.}
			\label{subfig:SPDImg:Result}
		\end{subfigure}
		\newline
		\begin{subfigure}{\textwidth}\centering
			\pgfplotstableread[col sep = comma]{Data/SPD/ImageCPPA-Cost.csv}\spdcppaData
			\pgfplotstableread[col sep = comma]{Data/SPD/ImageCP-Cost.csv}\spdcpData
			\pgfplotstableread[col sep = comma]{Data/SPD/ImageDR-Cost.csv}\spddrData
			\begin{tikzpicture}
				\begin{axis}[
					legend style={font=\scriptsize},
					xmode=log,
					log ticks with fixed point,
					width=.9\textwidth,
					height=.3\textwidth,
					xlabel={Iterations},
					xmin=1, xmax=4000,
					]
					\addplot[TolVibrantGray, mark=none, forget plot] coordinates { 
							(1, 38.73702639431868)
							(4000, 38.73702639431868)
						};
					\addplot[TolVibrantTeal, mark=none] table[x index = {0}, y index = {1}]{\spdcppaData};
					\addplot[TolVibrantBlue, mark=none] table[x index = {0}, y index = {1}]{\spddrData};
					\addplot[TolVibrantOrange, mark=none] table[x index = {0}, y index = {1}]{\spdcpData};
					\legend{{CPPA}, {PDRA}, {lRCPA}}
				\end{axis}
			\end{tikzpicture}
			\caption{Cost function.}
			\label{subfig:SPDImg:Cost}
		\end{subfigure}
	}{}
	\caption{Development of the three algorithms Cyclic Proximal Point (CPPA),
		parallel Douglas--Rachford (PDRA) as well as the linearized Riemannian
		Chambolle--Pock~\cref{alg:DualOverrelax-lRCP} (lRCPA) starting all from the original data
		in~(\subref{subfig:SPDImg:Orig}) reaching the final value (image)
		in~(\subref{subfig:SPDImg:Result}) is shown in~(\subref{subfig:SPDImg:Cost}),
	where the iterations on the x-axis are in log-scale.}
\end{figure}

\begin{figure}[htb]\centering
	\ifthenelse{\boolean{withpictures}}{%
		\begin{subfigure}{.32\textwidth}\centering
			\includegraphics[width=0.98\textwidth]{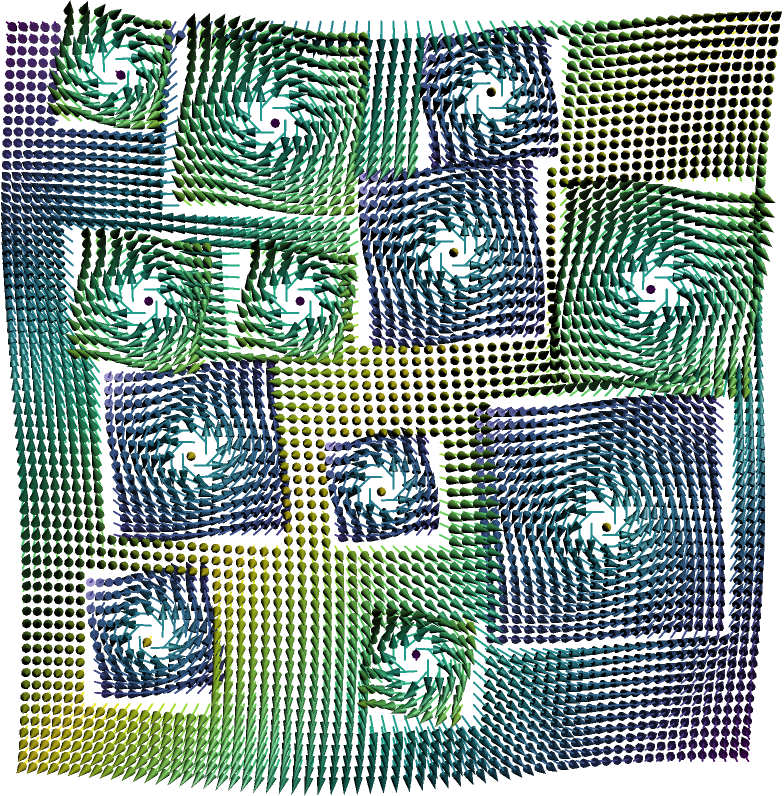}
			\caption{The original \\ \lstinline!S2Whirl! data.}
			\label{subfig:S2Whirl:orig}
		\end{subfigure}
		\begin{subfigure}{.32\textwidth}\centering
			\includegraphics[width=0.98\textwidth]{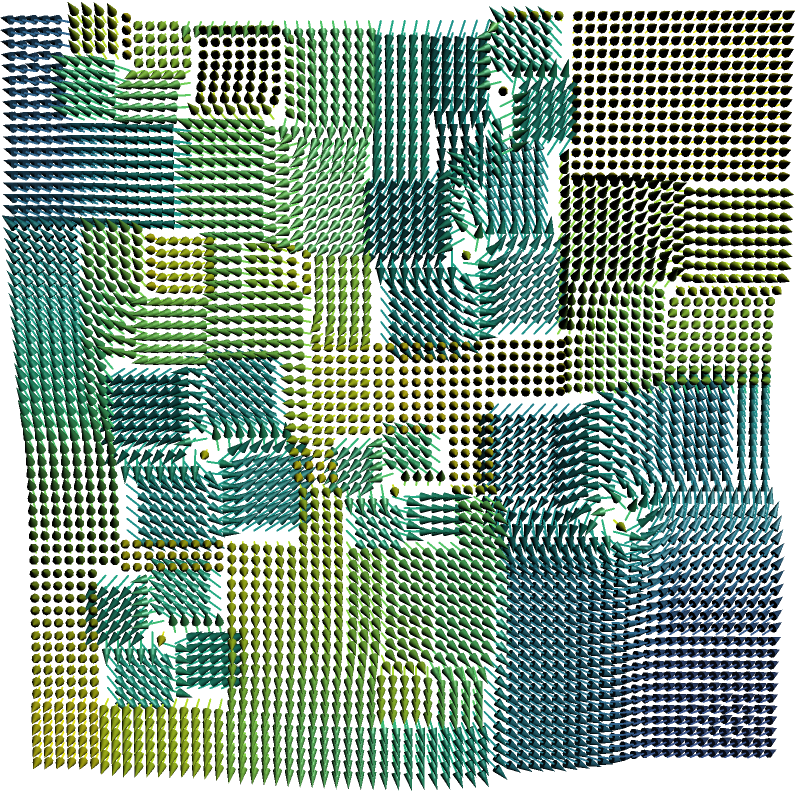}
			\caption{The result with \\ base $m$ mean.}
			\label{subfig:S2Whirl:mean}
		\end{subfigure}
		\begin{subfigure}{.32\textwidth}\centering
			\includegraphics[width=0.98\textwidth]{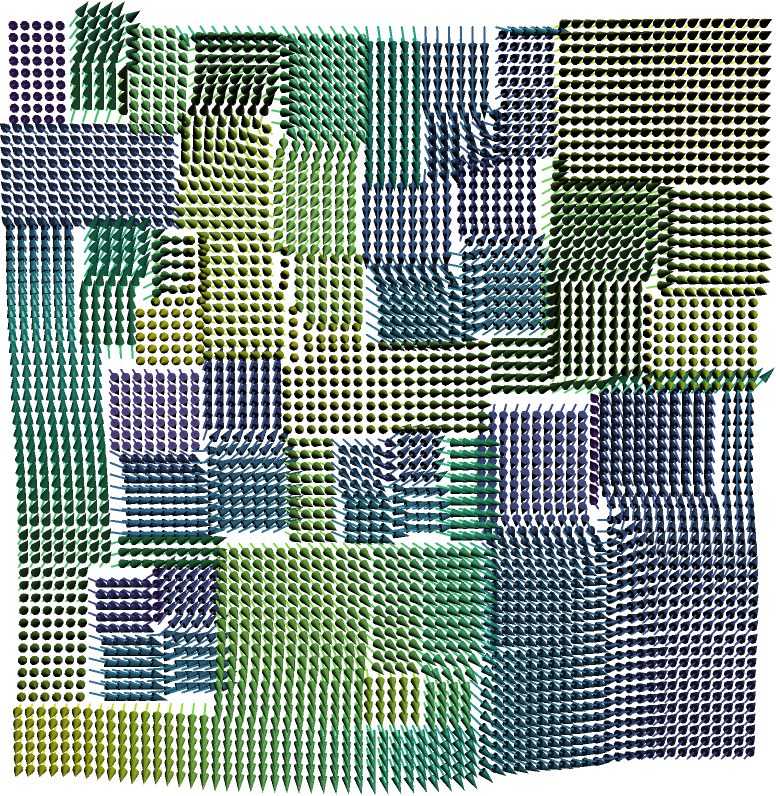}
			\caption{The result with \\ base $m$ west.}
			\label{subfig:S2Whirl:west}
		\end{subfigure}
		\newline
		\begin{subfigure}{\textwidth}
			\centering
			\pgfplotstableread[col sep = comma]{Data/S2/S2Whirl-Result.csv}\swhirlData
			\begin{tikzpicture}
				\begin{axis}[
					legend style={font=\scriptsize},
					log ticks with fixed point,
					width=.9\textwidth,
					height=.3\textwidth,
					xlabel={Iterations},
					xmin=1, xmax=300,
					]
					\addplot[TolVibrantTeal, mark=none] table[x=i, y=mMean]{\swhirlData};
					\addplot[TolVibrantBlue, mark=none] table[x=i, y=mWest]{\swhirlData};
					\legend{{mean}, {west}}
				\end{axis}
			\end{tikzpicture}
			\caption{cost function}
			\label{subfig:S2Whirl:cost}
		\end{subfigure}
	}{}
	\caption{The \lstinline!S2Whirl! example illustrates that for manifolds with positive curvature, the algorithm still converges quite fast, but due to the nonconvexity of the distance, the effect of the linearization influences the result.}
	\label{fig:S2Whirl}
\end{figure}
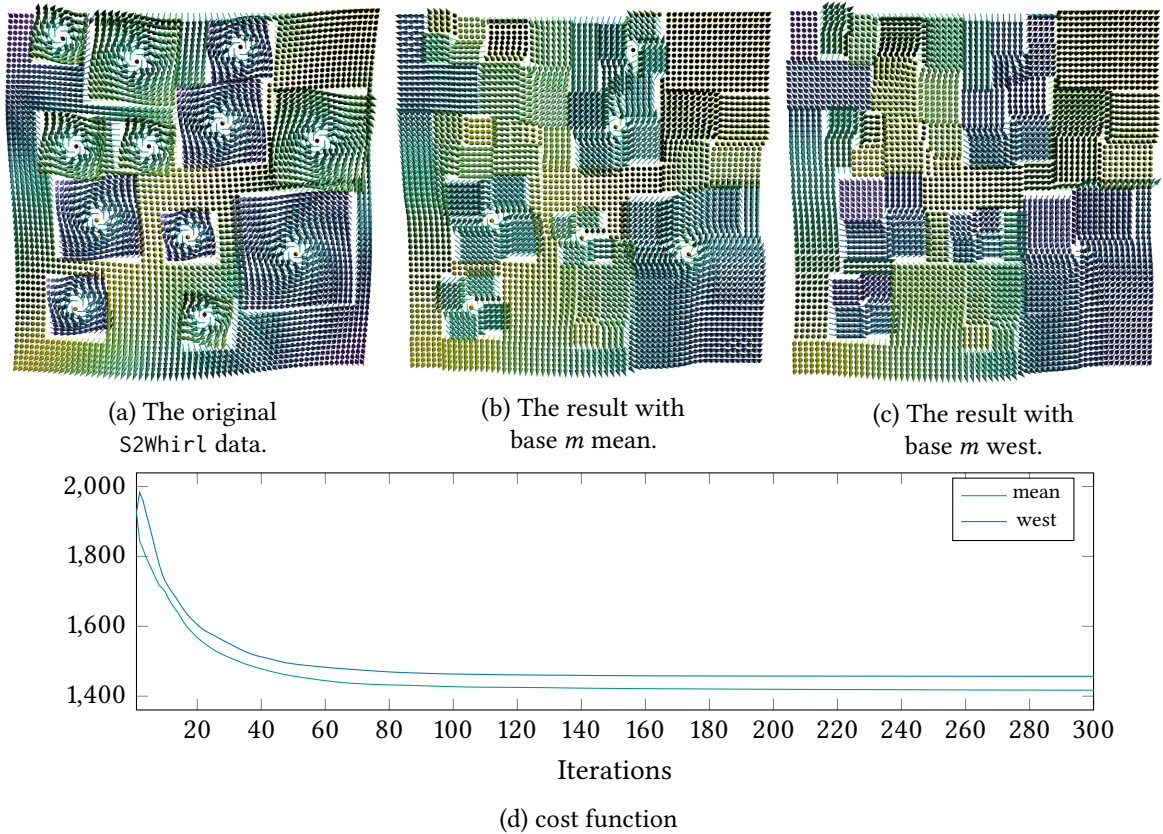

\subsection{Dependence on the Point of Linearization}

We mentioned previously that \cref{alg:DualOverrelax-lRCP} depends on the base points $m$ and $n$ and it cannot, in general, be expected to converge to a saddle point of~\eqref{eq:Manifold_saddle-point_representation_for_PD_Chambolle-Pock} since it is based on the \emph{linearized} saddle-point problem~\eqref{eq:Linearized_manifold_saddle-point_reprsentation_for_PD_Chambolle-Pock}.
In this experiment we illustrate the dependence of the limit of the sequence of primal iterates on the base point~$m$.

As data~$f$ we use the \lstinline!S2Whirl! image designed by Johannes Persch in~\cite{LausNikolovaPerschSteidl:2017:1}, adapted to \manoptjl, see~\cref{subfig:S2Whirl:orig}.
We set~$\alpha=1.5$ in the manifold-valued anisotropic $\ell^2$-TV model, \ie, \eqref{eq:M-LqTV} with $q = 1$.
We ran \cref{alg:DualOverrelax-lRCP} with $\primalstep = \dualstep = 0.35$ and $\gamma = 0.2$ for $300$~iterations.
The initial iterate is $p^{(0)}=f$ and $\xi^{(0)}_n$ as the zero vector.

We compare two different base points~$m$.
The first base point is the constant image whose value is the mean of all data pixels.
The second base point is the constant image whose value is $p=(1,0,0)^{\mathrm{T}}$ (\eqq{west}).
The final iterates are shown in~\cref{subfig:S2Whirl:mean} and \cref{subfig:S2Whirl:west}, respectively.
The evolution of the cost function value during the iterations is given in~\cref{subfig:S2Whirl:cost}.
Both runs yield piecewise constant solutions, but since their linearizations of~$\Lambda$ are using different base points, they yield different linearized models.
The resulting values of the cost function~\eqref{eq:M-LqTV} differ, but both show a similar convergence behavior.

\section{Conclusions}
\label{sec:Conclusions}

This paper introduces a novel concept of Fenchel duality for manifolds.
We investigate properties of this novel duality concept and study corresponding primal-dual formulations of non-smooth optimization problems on manifolds.
This leads to a novel primal-dual algorithm on manifolds, which comes in two variants, termed the exact and linearized Riemannian Chambolle--Pock algorithm.
The convergence proof for the linearized version is given on arbitrary Hadamard manifolds under a suitable assumption.
It is an open question whether condition \eqref{eq:Sufficient_conditions} can be removed.
The convergence analysis accompanies an earlier proof of convergence for a comparable method, namely the Douglas--Rachford algorithm, where the proof is restricted to Hadamard manifolds of constant curvature.
Numerical results illustrate not only that the linearized Riemannian Chambolle--Pock algorithm performs as well as state-of-the-art methods on Hadamard manifolds, but it also performs similarly well on manifolds with positive sectional curvature.
Note that here it also has to deal with the absence of a global convexity concept of the functional.

A more thorough investigation as well as a convergence proof for the exact variant are topics for future research.
Another point of future research is an investigation of the choice of the base points~$m\in\cM$ and~$n\in\cN$ on the convergence, especially when the base points vary during the iterations.

Starting from the proper statement of the primal and dual problem for the linearization approach of~\cref{subsec:Linearized Riemannian Primal-Dual Chambolle-Pock},
further aspects are open to investigation, for instance, regularity conditions ensuring strong duality.
Well-known closedness-type conditions are then available, opening in this way a new line of rich research topics for optimization on manifolds.

Another point of potential future research is the measurement of the linearization error introduced by the model from~\cref{subsec:Linearized Riemannian Primal-Dual Chambolle-Pock}.
The analysis of the discrepancy term, as well as its behavior in the convergence of the linearized algorithm \cref{alg:DualOverrelax-lRCP}, are closely related to the choice of the base points during the iteration, and should be considered in future research.

Furthermore, our novel concept of duality permits a definition of infimal convolution and thus offers a direct possibility to introduce the total generalized variation.
In what way these novel priors correspond to existing ones, is another issue of ongoing research.
Furthermore, the investigation of both a convergence rate as well as properties on manifolds with non-negative curvature are also open.

\paragraph{Acknowledgement}

The authors would like to thank two anonymous reviewers for their insightful comments which helped to improve the manuscript significantly.
RB would like to thank Fjedor Gaede and Leon Bungert for fruitful discussions concerning the Chambolle--Pock algorithm in $\R^n$, especially concerning the choice of parameters as well as DT for hospitality in Münster and Erlangen.
The authors would further like to thank Tuomo Valkonen for discussions on Hadamard manifolds and a three-point inequality remark, as well as Nicolas Boumal, Sebastian Neumayer, Gabriele Steidl for discussions and suggestions on preliminary versions of this manuscript.
RB would like to acknowledge funding by the \href{https://gepris.dfg.de/gepris/projekt/288750882}{DFG project BE~5888/2}.
DT would like to acknowledge support within the \href{https://www.uni-muenster.de/NoMADS/}{EU grant No.~777826, the NoMADs project}.
RH and JVN would like to acknowledge the \href{https://spp1962.wias-berlin.de}{Priority Program SPP~1962} (\emph{Non-smooth and Complementarity-based Distributed Parameter Systems: Simulation and Hierarchical Optimization}), which supported this work through the \href{https://gepris.dfg.de/gepris/projekt/314150341}{DFG grant HE~6077/10--1}.
MSL is supported by a measure which is co-financed by tax revenue based on the budget approved by the members of the Saxon state parliament.
Financial support is gratefully acknowledged.

\section*{Conflict of interest}

The authors declare that they have no conflict of interest.